\documentclass{article}

\usepackage{latexsym}
\usepackage{eucal}
\usepackage{amsmath}
\usepackage{amssymb}
\usepackage{ amscd}
\usepackage{mathrsfs}
\usepackage{amsthm}
\usepackage{bm}
\usepackage{quiver}

\usepackage[a4paper, total={6.5in, 8.5in}]{geometry}

\newtheorem{definition}{Definition}[section]
\newtheorem{theorem}{Theorem}[section]
\newtheorem{corollary}[theorem]{Corollary}
\newtheorem{lemma}[theorem]{Lemma}
\newtheorem{proposition}[theorem]{Proposition}

\usepackage{tikz-cd}
\usepackage{tikz}
\usetikzlibrary{matrix}
\usepackage[numbers]{natbib}

\usepackage[english]{babel}
\usepackage{hyperref}
\hypersetup{
    colorlinks=true,
    linkcolor=blue,
    filecolor=blue,      
    citecolor=blue,
    urlcolor=blue,
    pdftitle={Overleaf Example},
    pdfpagemode= FullScreen,
}
\usepackage{url}

\usepackage{graphicx}
\DeclareUnicodeCharacter{221E}{\ensuremath{\infty}}

\DeclareMathOperator{\Set}{\mathcal{S}et}
\DeclareMathOperator{\sSet}{s\mathcal{S}et}
\DeclareMathOperator{\semisSet}{\mathcal{S}emi\sSet}
\DeclareMathOperator{\Top}{\mathcal{T}op}
\DeclareMathOperator{\Cat}{\mathcal{C}at}
\DeclareMathOperator{\sCat}{s\mathcal{C}at}

\DeclareMathOperator{\colim}{\mathrm{colim}}

\newcommand{\scr}[1]{\mathcal{#1}}

\DeclareMathOperator{\Map}{\scr{M}\mathrm{ap}}
\DeclareMathOperator{\Hom}{\scr{H}\mathrm{om}}
\DeclareMathOperator{\Fun}{\scr{F}\mathrm{un}}
\DeclareMathOperator{\fgt}{\mathrm{fgt}}
\DeclareMathOperator{\free}{\mathrm{free}}
\DeclareMathOperator{\Aut}{haut}

\title{On the Classifying Space of Homogeneous Functors}
\author{Jiahao Li}

\usepackage{lipsum}
\usepackage{adjustbox}
\begin{document}

\maketitle
\begin{abstract}
    Let $M$ be a manifold and let $\mathcal{M}$ be a simplicial model category. Given an object $A$ in $\mathcal{M}$, Tsopméné and Stanley constructed a topological space $\hat{A}$ that classifies homogeneous functors of degree $k$ from the poset of open subsets of $M$ into $\mathcal{M}$. They showed that the set of weak equivalent classes of such functors that maps disjoint union of $k$ open balls to $A$ is in bijection with the set $[F_k(M), \hat{A}]$ of homotopy classes of maps out of $F_k(M)$, the unordered configuration space of $k$ points in $M$. In this paper, we begin a study of the space $\hat{A}$, and we prove
that $\hat{A}$ is weakly equivalent to the classifying space $B\Aut(A)$, where $\Aut(A)$ is the simplicial monoid of self weak equivalences of $A$. This proves a conjecture of Tsopméné and Stanley. Our result enables us to generalize the classification of homogeneous functors of Weiss for $\scr{M}=\Top$ to any simplicial model category.
\end{abstract}

\tableofcontents

\section{Introduction}

Motivated by the study of spaces of embeddings, Weiss introduced in \cite{Weiss_1999} the theory of \textit{Manifold Calculus of Functors}. The general idea is as follows. Let $\scr{O}(M)$ be the category of open subsets of $M$ with inclusions as morphisms and let $\scr{M}$ be a simplicial model category. Manifold Calculus of Functors is the study of contravariant functors $F:\scr{O}(M)^{op}\to \scr{M}$ through successive approximations of what one calls polynomial functors, similar as approximating smooth functions through Taylor polynomials. Much like each term of a Taylor polynomial is homogeneous, there is the concept of \textit{homogeneous functors}, which we will formally introduce in Section~\ref{section46}. Weiss in \cite{Weiss_1999} classified homogeneous functors of degree $k$ for $\scr{M}= \Top$ by fibrations. More precisely, let $F_k(M)$ be the space of unordered configurations of $k$ points in $M$. For any homogeneous functor $F$ of degree $k$, Weiss showed that there exists a fibration $p:Z\to F_k(M)$ such that for every $S\in F_k(M)$, $p^{-1}(S)\simeq F(U_S)$, where $U_S$ is any choice of tubular neighborhoods of $S$ that is diffeomorphic to the disjoint union of $k$ open balls. 
    
    Weiss's result is specific to $\scr{M}=\Top$. If $\scr{M}$ is any simplicial model category, a separate result is given by Tsopméné and Stanley in \cite{tsopmene2024classificationhomogeneousfunctorsmanifold}.  For any object $A\in \scr{M}$, let $\scr{F}_{kA}(\scr{O}(M)^{op}, \scr{M})/\mathfrak{w}$ be the collection of weak equivalence classes of homogeneous functors $F:\scr{O}(M)^{op}\to \scr{M}$ of degree $k$ such that $F(U)\simeq A$ for any $U\in \scr{O}(M)$ that is diffeomorphic to the disjoint union of $k$ open balls. In \cite{tsopmene2024classificationhomogeneousfunctorsmanifold}, a space $\hat{A}$ is constructed out of the object $A$, and it is proved that
    \begin{itemize}
        \item If $k=1$, then $\scr{F}_{1A}(\scr{O}(M)^{op}, \scr{M})/\mathfrak{w}\cong [M, \hat{A}]$
        \item If $k>1$ and $\scr{M}$ is pointed, then $\scr{F}_{kA}(\scr{O}(M)^{op}, \scr{M})/\mathfrak{w}\cong [F_k(M), \hat{A}]$
    \end{itemize}

The classification of Weiss and that of Tsopméné and Stanley use two different approaches. However, the results should be essentially identical. In this paper, we prove Theorem~\ref{main} below. Let $\Aut(A)$ be the simplicial monoid of self weak equivalences of $A$ and let $B\Aut(A)$ be its classifying space.
\begin{theorem}\label{main}
    Let $\scr{M}$ be a simplicial model category. For any fibrant-cofibrant object $A\in \scr{M}$, the space $\hat{A}$ is weakly equivalent to $B\Aut(A)$.
\end{theorem}
An immediate consequence of this theorem is the following. 
\begin{theorem}
\label{class}
Let $\scr{M}$ be a simplicial model category. For $k=1$ or $k>1$ and $\scr{M}$ is pointed, there is a bijection
\[
    \scr{F}_{kA}(\scr{O}(M)^{op}, \scr{M})/\mathfrak{w}\cong [F_k(M), B\Aut(A)]
\]
\end{theorem}This proves a conjecture of Tsopméné and Stanley \cite[Conjecture 8.3]{tsopmene2024classificationhomogeneousfunctorsmanifold}. 

Our proof strategy is to construct a zigzag of weak equivalence between $B\Aut(A)$ and $\hat{A}$:
\begin{itemize}
    \item The first weak equivalence, introduced in Section~\ref{section41}, provides a model of $BG$ for a simplicial monoid $G$ with some assumption. The classic approach is to take the diagonal of the Bar construction of $G$. However, if $G$ is a simplicial monoid that is also a Kan complex, the simplicial category with one object and the mapping space $G$ is equivalent to $BG$. For our concern, this monoid is $\Aut(A)$. Then, for any full subcategory of $\scr{M}^{cf}\cap \scr{W}$ with objects that are weakly equivalent to $A$ acts as a model of $B\Aut(A)$.
    \item The second weak equivalence, introduced in Section~\ref{section42}, connects the homotopy coherent nerve of a simplicial category $\scr{C}$ to its ordinary nerve. This is true for the category $\scr{C}$ when it satisfies certain property, see Proposition~\ref{ordinarynerve}. This property then becomes an anchor for us to generalize the result to every category of our concern. 
    \item The third weak equivalence, introduced in Section~\ref{section43}, is a special case of a well-known result in \cite[Chapter III, 4.]{goerss-jardine} that takes a simplicial set $X$ to a refined simplicial set $\mathrm{Ex}(X)$ of the same weak equivalence type.
    \item The last weak equivalence, introduced in Section~\ref{section44}, connects $\hat{A}$ to a more natural replacement of the space. The technicality of this section relies on the connection between the category of semisimplicial sets and the category of topological spaces.
\end{itemize}

There are several important consequences of Theorem~\ref{main}. First of all, as we mentioned above, the study of homogeneous functors provides a way to study the Taylor approximation of any good functor of our taste. There are already many applications of Manifold Calculus in the field of embedding spaces, as is its original purpose. For example, given another smooth manifold $N$ satisfying certain conditions, the embedding space functor $\mathrm{Emb}_N:=\mathrm{Emb}(\bullet, N):\scr{O}(M)^{op}\to \Top$ is a good functor, and its $1$st homogeneous layer $\scr{L}_1\mathrm{Emb}_N$ is equivalent to the functor of immersion $\mathrm{Imm}(\bullet, N)$ \cite[Page 97]{Weiss_1999}. Moreover, the study of $\Aut(A)$ and its classifying space is itself interesting. For instance, the study of the homotopy types of the classifying spaces of simplicial monoids tells us information about their (simplicial) group completion \cite{QuillenGroupComp}. The homotopy type of the space can be difficult to obtain. However, now with the connection of the homotopy type to the study of Manifold Calculus, we might be able to construct a theory of characteristic classes of homogeneous functors, just as the theory of characteristic classes of principal bundles. Therefore, the proof of Theorem~\ref{main} allows many possibilities for the connection of different fields and problems. 

This paper is structured as follows: In Section~\ref{section2}, we introduce the focus of this paper, the space $\hat{A}$ \cite{tsopmene2024classificationhomogeneousfunctorsmanifold}.  In Section~\ref{section3}, we state and prove some results that will be used in the proof of our main theorem. We do not claim originality for any of the results in this section, because they were first proven elsewhere. The purpose of this section is simply to provide a relatively modern approach to the results. In Section~\ref{section4}, we provide a proof of the main theorem and connects it to Manifold Calculus of Functors. In Section~\ref{section5}, we relate the classification of homogeneous functors in a general simplicial model category to that of Weiss. Lastly, in Section~\ref{section6}, we investigate the ∞-categorical analogy of Manifold Calculus developed by Arakawa in \cite{arakawa2026contextmanifoldcalculus} and its connection to ours in simplicial model categories.

\textbf{Acknowledgment}: This paper is the result of the author's Directed Studies at the University of British Columbia, Okanagan Campus, under the supervision of Dr. Paul Tsopméné, the first author of \cite{tsopmene2024classificationhomogeneousfunctorsmanifold}. We sincerely thank him for the guidance and support for the author and this paper, which without the paper would be impossible.

\section{The Space \protect\boldmath$\hat{A}$}\label{section2}

The purpose of this section is to recall the definition of the space $\hat{A}$ introduced in \cite{tsopmene2024classificationhomogeneousfunctorsmanifold}.

We begin with some notation.
\begin{itemize}
    \item $\Delta$, the category of nonempty finite totally ordered sets whose morphisms are order preserving functions.
    \item $\scr{P}(D)$, the category whose objects are subsets of $D$, for some finite set $D$, whose morphisms are the inclusions of subsets.
    \item $\partial\scr{P}(D)$, the full subcategory of $\scr{P}(D)$ without the empty set.
    \item    $\tilde\Delta^n$, the category whose objects are the nonempty totally ordered subsets of $[0<1<...<n]$, for any $n\geq 0$, whose morphisms are inclusions.
    \item   $\partial \tilde\Delta^n$, the full subcategory of $\tilde\Delta^n$ without the object $[n]$. 
    \item $\tilde\Lambda^n_k$, the full subcategory of $\partial \tilde{\Delta}^n$ without $[0<1<...<\hat{k}<...<n]$, where $\hat{k}$ means that we delete $k$.
\end{itemize}

    The space $\hat{A}$ is a parametrization of fibrant diagrams $(\tilde\Delta^{n})^{op}\to \scr{M}$ that have the same weak equivalence type, such that every continuous map $X\to \hat{A}$ determines a ``diagram'' that resembles the shape of $X$, obtained by gluing the diagrams parametrized by $\hat{A}$ together. More precisely, $\hat{A}$ is defined as follows.
    
   Given $f:X\to Y$, we denote $X\stackrel{\sim}{\rightarrowtail} V_f\twoheadrightarrow Y$ and $X{\rightarrowtail} W_f\stackrel{\sim}{\twoheadrightarrow}Y$ the functorial factorizations of $f$. Let $\scr{M}$ be a simplicial model category,

\begin{definition}
\label{CAcategory}   Let $A\in\scr{M}$ be fibrant and cofibrant. Define $\scr{C}_A$ to be the smallest full subcategory of $\scr{M}$ satisfying the following properties:
    \begin{itemize}
        \item $A\in \scr{C}_A$.
        \item For any object $X\in \scr{C}_A$, the cylinder object of $X$, $Z_X$, is in $\scr{C}_A$.
        \item For any object $X\in \scr{C}_A$, any $F: (\partial\tilde\Delta^{n})^{op}\to \scr{C}_A$ that is fibrant in the injective model structure, and any morphism $\phi:X\to\lim F$, $V_\phi\in \scr{C}_A$.
        \item For any $F: (\partial\tilde\Delta^{n})^{op}\to \scr{C}_A$ that is fibrant in the injective model structure, and for any $G:\partial \scr{P}(D)^{op}\to \scr{C}_A$ that is cofibrant in the projective model structure and maps every morphism to a weak equivalence, if there is $\phi:\colim G\to \lim F$, then $V_\phi\in\scr{C}_A$.
        \item $\scr{C}_A$ is closed under fibrant and cofibrant replacements, $Q$ and $R$.
        \item For any $F:(\tilde\Lambda^{n}_k)^{op}\to \scr{C}_A$ that is fibrant in the injective model structure, if $F$ maps every morphism to a weak equivalence, $Q\colim F\in \scr{C}_A$.
    \end{itemize}
\end{definition}

In fact, $\scr{C}_A$ is a small simplicial category in which all objects are fibrant-cofibrant and weakly equivalent to $A$ \cite[Lemma 3.1]{tsopmene2024classificationhomogeneousfunctorsmanifold}.

Notice that there exists a natural cosimplicial structure on the collection $\{\tilde\Delta^{n}\}_{n\geq0}$ in $\Cat$, induced by the order preserving functions in the category $\Delta$. For any $f:[n]\to[m]\in \Delta$, let $\tilde{f}:\tilde\Delta^{n}\to \tilde\Delta^{m}$ defined as 
\[\{a_0<a_1<...<a_k\}\mapsto \{f(a_0)\leq f(a_1)\leq ...\leq f(a_k)\}
\] We abuse notation by denoting the coface and the codegeneracy maps respectively as ${d^i}$ and ${s^i}$.

\begin{definition}
    Given any fibrant-cofibrant object $A\in\scr{M}$, the simplicial set $\hat{A}'_\bullet$ is defined as the subsimplicial set of $\Hom_{\Cat}((\tilde\Delta^{\bullet})^{op},\scr{C}_A)$ of weak equivalence diagrams. That is, for any $n$, the set $\hat{A}_n'$ is the set of functors $(\tilde\Delta^{n})^{op}\to \scr{C}_A$ that maps every morphism to a weak equivalence, where the simplicial structure on $\hat{A}'_\bullet$ is induced by the cosimplicial structure of $(\tilde\Delta^{\bullet})^{op}$ through precomposition.
\end{definition}
Therefore, the face and degeneracy structure on the simplicial set $\hat{A}'_\bullet$ comes naturally from the cosimplicial structure of $\tilde{\Delta}^\bullet$.

\begin{definition}
    For any $n$, define $\hat A_n$ to be the subset of $\hat{A}'_n$ of fibrant diagrams in the injective model structure.
\end{definition}

In \cite[Section 3.3]{tsopmene2024classificationhomogeneousfunctorsmanifold}, a simplicial structure is constructed on the collection $A_{\bullet} = \{\hat{A}_n\}_{n\geq 0}$. The face maps, $d_i: \hat{A}_n\to \hat{A}_{n-1}$, are defined by the precomposition $d_i(F):=F\circ d^{i}$. This is well defined because if $F$ is fibrant, then every face of $F$ is necessarily fibrant.  The construction of $s_i:\hat{A}_n\to \hat{A}_{n+1}$ is complicated and we refer the reader to \cite[Section 3.3]{tsopmene2024classificationhomogeneousfunctorsmanifold}. We will denote $\hat{A}:=|\hat{A}_\bullet|$, the geometric realization of $\hat{A}_\bullet$.

\section{Preliminary Results}\label{section3}

This section is dedicated to the results of others that will be used in this paper. Some of the results are fairly older and were written in a notation that could be unfamiliar to those who are used to the modern notation. Therefore, we rewrite these results in the modern notation that most would recognize. Some of the results are considerably less well-known than other results used in this paper and the length of the details deserve their own sections.

\subsection{Properties of Semisimplicial Sets}
\label{section31}

The connection between the categories $\semisSet$  and $\sSet$ was well studied in \cite{Rourkesemisimplicialset}. However, many of the notations used in their work would be ambiguous to modern readers. 

Let $\Delta^\bullet:\Delta\to \Top$ be the standard cosimplicial space such that $\Delta^n\subset \mathbb{R}^{n+1}$ is the standard $n$-simplex. 

Let $|-|: \sSet \to \Top$ be the geometric realization functor of simplicial sets, which is defined objectwise as $|X|:= \coprod_{n}X{(n)}\times\Delta^n/\sim$, where $\sim$ is the equivalence relation generated by the relation: for any $(\sigma, x)\in X{(n)}\times\Delta^n$ and $(\sigma',x')\in X{(m)}\times\Delta^m$, $(\sigma, x)\sim (\sigma',x')$ if and only if there is $\lambda:[n]\to [m]\in \Delta$ such that 
\[\sigma=X(\lambda)(\sigma')\text{ \ and }x'=\Delta^\bullet(\lambda)(x)
\]We will denote the points of $|X|$ as $[\sigma,x]$ for $\sigma \in X(n)$ and $x\in \Delta^n$.

\begin{lemma}
    \cite[Lemma 1.1] {Rourkesemisimplicialset} \label{factor}
    For any $\lambda:[n]\to[k]\in \Delta$, there exists a unique factorization of $\lambda$, $\lambda=\lambda_1\circ\lambda_2$, where $\lambda_1$ is injective and $\lambda_2$ is surjective.
\end{lemma}
\begin{definition}
   Define $\Delta_\mathrm{inj}$ to be the subcategory of $\Delta$ with the same objects, but the morphisms are the functions that are injective. 

    The category of semisimplicial sets, $\semisSet$, is the functor category $\Fun(\Delta_\mathrm{inj}^{op},\Set)$. Let $\fgt:\sSet\to \semisSet$ be the functor induced by the inclusion $\Delta_\mathrm{inj}\hookrightarrow \Delta$.
    
\end{definition}

There is a functor $\free:\semisSet\to \sSet$, defined as follows. For $Y\in \semisSet$, \begin{align*}
    \free(Y){(n)}:=\{(\sigma, \mu)\ |\ \forall k\geq 0,   \mu:[n]\to [k] \text{ is surjective, }\sigma\in Y{(k)}\}
\end{align*}
And for any morphism $\lambda:[m]\to [n]\in \Delta$, 
\begin{align*}
    \free(Y)(\lambda):\free(Y){(n)}\to  \free(Y){(m)}
    \\(\sigma, \mu)\mapsto (Y(\phi_1)(\sigma), \phi_2)
\end{align*} where $\phi_1\circ \phi_2=\mu\circ \lambda$ is the unique factorization in the previous lemma.

For $f:Y\to X \in \semisSet$ define $\free (f):\free(Y)\to \free(Z)$, as $\free(f)(\sigma, \mu)=(f(\sigma),\mu)$.

\begin{proposition}
    \cite[Theorem 1.7]{Rourkesemisimplicialset} \label{adj}The pair $(\free,\fgt)$ determines an adjunction $\free \dashv \fgt: \semisSet\to \sSet$.
\end{proposition}

For $Y\in \semisSet$, let $\sim_{\mathrm{semi}}$ be the equivalence relation on $\coprod_nY{(n)}\times\Delta^n$ generated by $(\sigma', x')\sim_\mathrm{semi}(\sigma,x) $ if and only if 
there is a morphism $\lambda\in \Delta_\mathrm{inj}$ such that $\sigma=Y(\lambda)(\sigma')$ and $x'=\Delta^\bullet(\lambda)(x)$. Define 
\begin{align*}
    \|Y\|:=\coprod_{n}Y{(n)}\times\Delta^n/\sim_\mathrm{semi}
\end{align*}We refer to this construction as the \textit{fat realization} of $Y$.

We denote a point in $\|Y\|$ as $[\sigma, x]_\mathrm{semi}$ for $\sigma\in Y(n)$ and $x\in \Delta^{n}$.
And for $f: Y\to Z\in \semisSet$, define $\|f\|:\|Y\|\to \|Z\|$ as $\|f\|[\sigma,x]_\mathrm{semi}=[f(\sigma),x]_\mathrm{semi}$, which is well-defined and continuous.

    Given $Y\in \semisSet$, define $\eta_Y:|\free(Y)|\to \|Y\|$ by
    \begin{align*}
        \eta_Y[(\sigma, \mu), x]=[\sigma, \Delta^\bullet(\mu)(x)]_\mathrm{semi}
    \end{align*}The map $\eta_Y$ is indeed well-defined. Let $[(\sigma', \mu'),x']_\mathrm{semi}\in \|\free (Y)\|$ such that there exists a morphism $\lambda \in \Delta$ where
    \begin{align*}
        \free(\lambda)(
    \sigma', \mu')=(Y(\phi_1)(\sigma'), \phi_2)=(\sigma, \mu)
    \end{align*}where $\phi_1\circ\phi_2=\mu'\circ \lambda$ is the unique factorization in Lemma~\ref{factor}, and $x'=\Delta^\bullet(\lambda)(x)$. Thus, \begin{align*}
        \eta_Y[(
    \sigma', \mu'),x']=[\sigma',\Delta^\bullet(\mu')(x')]_\mathrm{semi}
    =[\sigma',\Delta^\bullet(\mu'\circ \lambda)(x)]_\mathrm{semi}
    \\=[\sigma',\Delta^\bullet(\phi_1\circ \phi_2)(x)]_\mathrm{semi}
\\= [Y(\phi_1)(\sigma'),\Delta^\bullet(\phi_2)(x)]_\mathrm{semi}
   \\=\eta_Y[(\sigma, \mu), x]
    \end{align*}

\begin{proposition}
    \cite[Proposition 2.1, part II]{Rourkesemisimplicialset} \label{homeo} For every $Y\in \semisSet$, $\eta_Y:  |\free(Y) | \to \|Y\| $ is a homeomorphism.
\end{proposition}
    
\begin{proof}
    Define $\nu_Y:\|Y\|\to |\free(Y)|$ by $\nu_Y[\sigma, x]_\mathrm{semi}=[(\sigma, id), x]$. For any $[\sigma',x']_\mathrm{semi}\in \|Y\|$ such that there exists a morphism $\lambda\in \Delta_\mathrm{inj}$ where $\sigma=Y(\lambda)(\sigma')$ and $x'=\Delta^\bullet(\lambda)(x)$,  
   \begin{align*}
       \nu_Y[\sigma',x']_\mathrm{semi}=[(\sigma', id),\Delta^\bullet(\lambda)(x)]=[(Y(\lambda)(\sigma'),id),x]=[(\sigma,id),x]=    \nu_Y[(\sigma,x)]_\mathrm{semi}
   \end{align*} Thus $\nu_Y$ is well-defined.
   
   Moreover, we have 
    \begin{align*}
        \eta_Y\circ\nu_Y[\sigma, x]_\mathrm{semi}=\eta_Y[(\sigma,id), x]=[\sigma, x]_\mathrm{semi}
        \\\nu_Y\circ \eta_Y[(\sigma, \mu), x]=\nu_Y[\sigma, \Delta^\bullet(\mu)( x)]_\mathrm{semi} = [(\sigma, id),\Delta^\bullet(\mu)(x)]
    \end{align*} Since $\mu$ is by definition surjective, 
     \begin{align*}
         (\sigma, id,\Delta^\bullet(\mu)(x))\sim  (\free(\mu)(\sigma, id),x)
         =((\sigma, \mu), x)
     \end{align*} Thus $(\sigma, id,\Delta^\bullet(\mu)(x))  \sim_\mathrm{semi}((\sigma, \mu), x)$.  Therefore, $\nu_Y$ and $\eta_Y$ are continuous inverses of each other.
\end{proof}
For any semisimplicial set $Y\in \semisSet$, we can construct a chain complex $C_*^{\mathrm{semi}}(Y)$ such that $C_n^{\mathrm{semi}}(Y):=\mathbb{Z}[ Y(n)]$, the free abelian group generated by $Y(n)$, and 
\begin{align*}
\partial_{n}:  C_n^{\mathrm{semi}}(Y)\to  C_{n-1}^{\mathrm{semi}}(Y)
    ,\ \sigma\mapsto \sum_{i=0}^n(-1)^i d_i\sigma
\end{align*}
Notice that if $Y=\fgt X$ for some $X\in \sSet$, there is a chain isomorphism $   C_*(X)\to C_*^{\mathrm{semi}}(Y)$.
\begin{lemma}
   \label{chain} For any $Y \in \semisSet$, the homology of $\|Y\|$ is isomorphic to the homology of $C_*^{\mathrm{semi}}(Y)$.
\end{lemma}
\begin{proof}
Consider the homeomorphism $\eta_Y:|\free(Y)|\to \|Y\|$ from Proposition~\ref{homeo}. The homology of $|\free(Y)|$ is calculated by $C_*(\free (Y))$. An $n$-simplex $(\sigma, \mu)\in \free(Y){(n)}$ is degenerated if and only if $\mu\neq id$ because $(\sigma, id)\in\free (Y){(k)}$ and $\free(Y)(\mu)(\sigma, id)=(\sigma, \mu)$, as $\mu $ is surjective. Therefore, the normalized chain complex, $ N_*(\free (Y))$, is the chain complex where for each grade $k$, $N_k(\free (Y))$ is the free abelian group generated by $(\sigma, id)$ where $\sigma\in Y{(k)}$. It is well-known that $C_*(\free (Y))\simeq N_*(\free (Y))$, and $N_*(\free (Y))\cong C_*^{\mathrm{semi}}(Y)$. This concludes the claim.
\end{proof}

Let $X\in \sSet$, define a continuous map $\kappa_X:|X|\to \|\fgt(X)\|$ as follows:

For every $\sigma\in X{(n)}$, there is a unique surjective function $\mu_{\sigma}:[n]\to [k]$ and a unique nondegenerate $\tau_{\sigma}\in X{(k)}$ such that $X(\mu_{\sigma})(\tau_{\sigma})=\sigma$. Define $\kappa_X[\sigma, x]=[\tau_{\sigma}, \Delta^\bullet(\mu_{\sigma})(x)]_\mathrm{semi}$. There is also an injective function $\tilde\mu_{\sigma}$ such that $\mu_{\sigma}\circ \tilde\mu_{\sigma}=id$.
If $(\sigma', x')\in |X|$ such that there is a morphism $\lambda \in\Delta$ where $X(\lambda)(\sigma')=\sigma$ and $\Delta^\bullet(\lambda)(x)=x'$. Then $X(\mu_{\sigma'}\circ\lambda)(\tau_{\sigma'})=X(\mu_\sigma)(\tau_\sigma)$, and thus 
\begin{align*}
X(\tilde\mu_\sigma)   (X(\mu_{\sigma'}\circ\lambda)(\tau_{\sigma'}))= X(\mu_{\sigma'}\circ \lambda\circ\tilde\mu_\sigma)(\tau_{\sigma'})=X(\mu_\sigma\circ \tilde{\mu}_\sigma)(\tau_\sigma)=\tau_\sigma
\end{align*}Therefore, $\mu_{\sigma'}\circ \lambda\circ \tilde\mu_{\sigma}$ is injective or else $\tau_\sigma$ is degenerate, a contradiction, so there is 

\begin{align*}
    \kappa_X[\sigma, x]=[\tau_\sigma,\Delta^\bullet(\mu_{\sigma})(x)]_\mathrm{semi}
    \\=[X(\mu_{\sigma'}\circ \lambda\circ\tilde\mu_\sigma)(\tau_{\sigma'}),\Delta^\bullet(\mu_{\sigma})(x) ]_\mathrm{semi}
 \\  =[\tau_{\sigma'},\Delta^\bullet(\mu_{\sigma'}\circ \lambda\circ \tilde \mu_\sigma\circ \mu_\sigma)(x)]_\mathrm{semi}
\end{align*}in $\|\fgt(X)\|$. 

In $|X|$, we have that 
\begin{align*}
    (\tau_\sigma, \Delta^\bullet(\mu_{\sigma})(x))\sim (\sigma, x)\sim (\sigma', x')\sim(\tau_{\sigma'},\Delta^\bullet(\mu_{\sigma'})(x'))
\end{align*} and the above shows that 
\begin{align*}
    (\tau_\sigma, \Delta^\bullet(\mu_{\sigma})(x))\sim (\tau_{\sigma'},\Delta^\bullet(\mu_{\sigma'}\circ \lambda\circ \tilde \mu_\sigma\circ \mu_\sigma)(x))
\end{align*} Therefore, we have 
\begin{align*}
    [\tau_{\sigma'},\Delta^\bullet(\mu_{\sigma'})(x')]=[\tau_{\sigma'},\Delta^\bullet(\mu_{\sigma'}\circ \lambda\circ \tilde \mu_\sigma\circ \mu_\sigma)(x)]\in |X|
\end{align*}
Since $\tau_{\sigma'}$ is nondegenerate, this implies that $\Delta^\bullet(\mu_{\sigma'})(x')=\Delta^\bullet(\mu_{\sigma'}\circ \lambda\circ \tilde \mu_\sigma\circ \mu_\sigma)(x)$. Therefore, 
\begin{align*}
       \kappa_X[\sigma, x]  =[\tau_{\sigma'},\Delta^\bullet(\mu_{\sigma'}\circ \lambda\circ \tilde \mu_\sigma\circ \mu_\sigma)(x)]_\mathrm{semi}
       \\=    [\tau_{\sigma'},\Delta^\bullet(\mu_{\sigma'})(x')]_\mathrm{semi}
       \\=\kappa_X[\sigma', x']
\end{align*}
So $\kappa_X$ is well-defined.

    Define $\iota_X:  \|\fgt (X)| | \to |X| $ by $\iota_X[\sigma, x]_\mathrm{semi}=[\sigma,x]$. It is not hard to see this is well-defined. Observe that $\iota_X\circ \kappa_X=id_{|X|}$, since $[\sigma,x]=[\tau_{\sigma},\Delta^\bullet(\mu_\sigma)(x)]$. The following result is crucial for this paper, and we will provide a  detailed proof of it.
\begin{proposition}
     \cite[Proposition 2.1, part I]{Rourkesemisimplicialset}\label{fgtsimeq} For every $X\in \sSet$, the map $\iota_X:  \|\fgt (X)| | \to |X| $ is a homotopy equivalence.
\end{proposition}
\begin{proof}
   From the previous lemma, $\iota_X$ is a homotopy equivalence if and only if $\iota_X\circ\eta_{\fgt (X)}:|\free \fgt(X)|\to |X|$ is. From Lemma~\ref{chain}, the induced chain complex homomorphism $C_*(\free \fgt(X))\to C_*^{\mathrm{semi}}(\fgt(X))\cong C_*(X)$ is a quasi-isomorphism.

    Now consider the fundamental groups. From the Cellular Approximation Theorem \cite[Theorem 4.8]{Hatcher:478079}, it suffices to consider only the cellular maps. Given a CW complex $X$, we will denote $X_n$ its $n$-skeleton.
    
    First, $\pi_1\iota_X$ is surjective because $\pi_1\iota_X\circ \pi_1\kappa_X=id_{\pi_1}$. Given two $\gamma, \gamma':(S^1, 0)\to(\|\fgt(X)\|_1, *)$ such that $\iota_X\gamma\simeq \iota_X\gamma'$, then $\kappa_X\iota_X\gamma\simeq \kappa_X\iota_X\gamma'$. Therefore, if $\gamma\simeq \kappa_X\iota_X\gamma$ for any $\gamma$, then $\iota_X$ is monic.
     
     For any $\theta \in S^1$, there is a $\sigma \in X(1)$ such that $\gamma(\theta)=[\sigma, x]_\mathrm{semi}\in \|\fgt(X)\|$. Then $\kappa_X\iota_X\gamma(\theta)$ is $[\tau_{\sigma}, \Delta^\bullet(\mu_\sigma)(x)]_\mathrm{semi}$. If $\sigma$ is degenerate, $d_0\sigma=d_1\sigma=\tau_\sigma$. Let $\Sigma:=s_1\sigma$, then a simple calculation shows
    \begin{align*}
 d_0\Sigma=d_1\Sigma=d_2\Sigma=\sigma
    \end{align*}
Thus, there is a subspace of $\| \fgt(X)\|$,
    \begin{align*}
        \|\Sigma\|:=\Sigma\times \Delta^2\coprod\sigma\times \Delta^1 \coprod \tau_\sigma\times \Delta^0/  \sim_\mathrm{semi}\subseteq   \|\fgt(X)\|
    \end{align*}

    Now the subspace $\|\Sigma\|$ can be viewed as a CW complex with $0$-skeleton being the point $\tau_\sigma$, $1$-skeleton a loop $S^1$, $2$-skeleton and above being $D^2\cup_f S^1$, where $f:\partial D^2\to S^1$ is the identification of $\partial \Sigma$ onto $\sigma$, which is homotopic to $id:S^1\to S^1$. Therefore, $\|\Sigma\|\simeq D^2\simeq *$. Since the interval of $I\subseteq S^1$ such that $\gamma(\theta)=[\sigma, x]_\mathrm{semi}$ factors through $ \|\Sigma\|$, this allows us to construct a homotopy $\gamma\simeq \kappa_X\iota_X\gamma$. Therefore, $\pi_1\iota_X$ is injective.

    Then, it follows from Whitehead's Theorem \cite[Theorem 4]{WhiteheadI} that a continuous map that induces isomorphisms on all homology groups and the fundamental group is a homotopy equivalence.
\end{proof}
This is a key result that we will use in Section~\ref{section44}.

We close this subsection with another result that we will use later. We first introduce some definitions.

\begin{definition}
    \cite[Definition, page 323]{Rourkesemisimplicialset}
Given $X\in \sSet$, define $\mathrm{Core}(X)\in \semisSet$ as follows
\begin{align*}
 \mathrm{Core}(X){(n)}=\{ \sigma\in X{(n)}\ |\ \exists \text{ an injection }\lambda:[n]\to [k] \text{ and a nondegenerate $k$-simplex }\tau
 \\ \text{such that } X(\lambda)\tau=\sigma\}
\end{align*}
And for $d_i:[n-1]\to [n]$, for $\sigma\in \mathrm{Core}(X){(n)}$, $d_i^{\mathrm{core}}\sigma :=d_i\sigma$.

 Given $X\in \sSet$, $X$ is n.d.c. if $\sigma\in X(n)$ is nondegenerate, then $d_i\sigma$ is nondegenerate for all $0\leq i\leq n$. Given two n.d.c. simplicial sets $X$ and $Y$, a simplicial map $f:X\to Y$ is n.d.c. if for any nondegenerate $\sigma\in X(n)$, $f(\sigma)$ is nondegenerate.
\end{definition}

Although $\mathrm{Core}$ is not a functor $\sSet\to \semisSet$ (look at the simplicial map $\Delta[1]\to \Delta[0]$), if $f:X\to Y$ is an n.d.c. simplicial map, there is a semisimplicial map $\mathrm{Core}(f):\mathrm{Core}(X)\to  \mathrm{Core}(Y)$ defined by $\mathrm{Core}(f)(\sigma):=f(\sigma)$. In particular, for any $n$ and any $0\leq k\leq n$, $i^n_k:\Lambda^n_k\to \Delta[n]$ is n.d.c., so there is a map $\mathrm{Core}(i^n_k):\mathrm{Core}(\Lambda^n_k)\to \mathrm{Core}(\Delta[n])$. We have the following result.
\begin{lemma}
\cite[Proposition 1.5]{Rourkesemisimplicialset} \label{phiX}For every n.d.c. $X\in \sSet$, there is an isomorphism \[\varphi_X: \free\circ \mathrm{Core}(X)\to X\]
\end{lemma}
Thus, for any n.d.c. $X\in \sSet$ and any $Y\in \sSet$, from the Lemma~\ref{phiX} and Proposition~\ref{adj}, there is an isomorphism $\Hom(\mathrm{Core}(X),\fgt(Y))\cong\Hom(\free\circ \mathrm{Core}(X),Y)\cong \Hom(X, Y)$.
 
\subsection{Homotopy Coherent Nerve of Small Simplicial Categories}
The homotopy coherent nerve functor $\scr{N}^{\mathrm{hc}}:\sCat\to \sSet$ was first introduced in \cite{Cordier1982DiagrammeHomotopique}. In this section, we recall the definition and state some of the results that will be used later.

\begin{definition}
    Given a totally ordered set $[n]:=\{0<1<...<n\}$, the simplicial category $\mathrm{Path}[n]$ is defined as follows:
    \begin{itemize}
        \item The objects of $\mathrm{Path}[n]$ are the elements of $[n]$.
        \item Given two objects $i$ and $j$, the mapping space is defined by $\Map(i,j):=\scr{N}(P_{i,j})$, where $\scr{N}$ is the ordinary nerve functor, and $P_{i,j}$ is the full subcategory of the opposite poset $P([n])^{op}$ whose objects are subsets $S$ with $\min S=i$ and $\max S= j$, for $i\leq j$. Otherwise, $\Map(i,j):=\emptyset$.
        \item The composition $\Map(i,j)\times \Map(j,k)\to \Map(i,k)$ is the nerve of the functor $P_{i,j}\times P_{j,k}\to P_{i,k}$ where\begin{align*}
            (S\in P_{i,j}, T\in P_{j,k})\mapsto S\cup T\in P_{i,k}
        \end{align*}

    \end{itemize}
The collection $\mathrm{Path}[\bullet]=\{\mathrm{Path}[n]\}_{n\geq 0}$ is equipped with a cosimplicial structure that is constructed as follows: for any order preserving function $\lambda:[n]\to [m]$, define simplicial functor \[ \mathrm{Path}(\lambda):\mathrm{Path}[n]\to \mathrm{Path}[m]\] that takes object $i$ to $\lambda(i)$ and
\begin{align*}
    \mathrm{Path}(\lambda)_{i,j}:P_{i,j}\to P_{\lambda(i),\lambda(j)}
    \\ S\subseteq \{i<...<j\}\mapsto \lambda(S)\subseteq \{\lambda(i)<...<\lambda(j)\}
\end{align*}
\end{definition}

\begin{definition}
    \cite{Cordier1982DiagrammeHomotopique} The homotopy coherent nerve functor is defined as
    \begin{align*}
\scr{N}^{\mathrm{hc}}:=\Hom_{\sCat}(\mathrm{Path}[\bullet],-):\sCat\to \sSet
    \end{align*} There exists a left adjoint of $\scr{N}^{\mathrm{hc}}$,  which is called the rigidification functor. Following \cite{lurieHTT}, we denote this left adjoint as $\mathfrak{C}$. Notice that $\mathfrak{C}\Delta[n]=\mathrm{Path}[n]$.
\end{definition}

Given a simplicial category $\scr{C}$ and two morphisms $f, g:X\to Y$ in $\scr{C}$, $f$ is \textit{homotopic} to $g$, and denote $f\simeq g$ if there is a $1$-simplex, $H\in \Map(X,Y)(1)$ such that $d_1H=f$ and $d_0H=g$. The map $f:X\to Y$ is a homotopy equivalence if there is a $g:X\to Y$ such that $g\circ f\simeq id_X$ and $f\circ g\simeq id_Y$.

Given a simplicial set $C\in \sSet$, an edge $f\in \scr{C}(1)$ is an\textit{ isomorphism} if there are two $2$-simplices that resemble the following:
\begin{center}
     \begin{tikzcd}
                                     & y \arrow[rd, "g"] &   &  & y \arrow[rd, "g"] \arrow[rr, "s_0y"] &                   & y \\
x \arrow[ru, "f"] \arrow[rr, "s_0x"] &                   & x &  &                                      & x \arrow[ru, "f"] &  
\end{tikzcd}
\end{center}

A simplicial category is \textit{locally Kan} if its mapping spaces are Kan complexes.
\begin{lemma}
    \label{lemma1} Given a locally Kan category $\scr{C}$, the $2$-simplices of $\scr{N}^{\mathrm{hc}}\scr{C}$ are characterized by three objects $x,y,z$, three morphisms $f:x\to y$, $g:y\to z$ and $h:x\to z$, such that $g\circ f\simeq h$ via a homotopy $H\in \Map(x,y)(1)$. In particular, a morphism in $\scr{C}$ is a homotopy equivalence if and only if it is an isomorphism in $\scr{N}^\mathrm{hc}\scr{C}$.
\end{lemma}
\begin{proof}
    The simplicial category $\mathrm{Path}[2]$ is composed of the following, 
    \begin{align*}
        \Map(0,1)=\{[0<1]\}\cong \Delta[0]
        \\\Map(1,2)=\{[1<2]\}\cong \Delta[0]
        \\\Map(0,2)=\scr{N}\{[0<1]\hookrightarrow[0<1<2]\}\cong \Delta[1]
    \end{align*} and $[0<1]\circ[1<2]=[0<1<2]\in\Map(0,2)$. 
 For a homotopy equivalence $f$, its homotopy inverse $g$, and the homotopies $H_X:g\circ f\Rightarrow id_X$ and $H_Y:f\circ g\Rightarrow id_Y$ determine two $2$-simplices.
\end{proof}
\begin{lemma}
   \label{Kancomplex} If $\scr{C}$ is a locally Kan category such that all morphisms are homotopy equivalences, then $\scr{N}^\mathrm{hc}\scr{C}$ is a Kan complex.
\end{lemma}
\begin{proof}
  Suppose $\scr{C}$ is a locally Kan category. Therefore, $\scr{N}^{\mathrm{hc}}\scr{C}$ is an ∞-category. By Lemma~\ref{lemma1}, every edge in $\scr{N}^{\mathrm{hc}}\scr{C}$ is an isomorphism. It follows from \cite[Corollary 1.4]{JOYAL2002207} that $\scr{N}^{\mathrm{hc}}\scr{C}$ is a Kan complex.
\end{proof}
\subsection{The Edgewise Subdivision of Simplicial Sets}\label{esd}

This section is dedicated to the functor $\mathrm{esd}^*:\sSet\to \sSet$ and some of its related constructions. This functor is well studied in \cite{KBergEsd}. However, many of the proofs are done in topology, and we would like the proof to be constructed in the simplicial setting for uniformity and simplicity. Therefore, we will prove some similar results using the simplicial homotopy theory, and these results will be used later.

We begin by recall the definition of the opposite of a simplicial set.
\begin{definition}
    Define the functor $\mathrm{op}:\Delta\to \Delta$ as follows:
\begin{align*}
         [n]=\{0<1<...<n\}\mapsto[n]=\{\widehat n<...<\widehat{1}<\widehat{0}\}
        \\ (f:[n]\to [m])\mapsto (\widehat{f}:[n]\to [m])
\end{align*} such that $\widehat{f}(\widehat{k})=\widehat{(f(k))}$.

    This induces a functor $(\bullet)^{op}:\sSet\to \sSet$ by precomposition. Given $X\in \sSet$, we refer to $X^{op}$  as the opposite of $X$. 
\end{definition}

Notice that the simplices of $X$ and $X^{op}$ are in bijection. However, they differ in the order of the faces and degeneracies. Moreover, given an ordinary category $\scr{C}\in \Cat$, $\scr{N}(\scr{C})^{op}\cong \scr{N}(\scr{C}^{op})$, because $\scr{N}([n]^{op})\cong \scr{N}([n])^{op}\cong \Delta[n]^{op}$.

We wish to show that $X$ is weakly equivalent to $X^{op}$, and we do this by constructing a zigzag of weak equivalences. One of which would involve the following construction.
\begin{definition}
   Define the functor $\mathrm{esd}:\Delta\to \Delta$ as follows:
    \begin{align*}
        [n]=\{0<1<...<n\}\mapsto[2n+1]=\{\widehat n<...<\widehat{1}<\widehat{0}<0<1<...<n\}
        \\ (f:[n]\to [m])\mapsto (\widehat{f}:[2n+1]\to [2m+1])
    \end{align*}such that $\widehat{f}({k})={f(k)}$ and $\widehat{f}(\widehat{k})=\widehat{(f(k))}$.

    This induces a functor, denoted $\mathrm{esd}^*:\sSet\to \sSet$, by precomposition. Given $X\in \sSet$, we refer to $\mathrm{esd^*}(X)$  as the edgewise subdivision of $X$.
\end{definition}
Note that $\mathrm{esd}$ determines a cosimplicial object in simplicial set $\mathrm{esd}[\bullet]$ such that
\begin{align*}
  \mathrm{esd}[\bullet]:\Delta\xrightarrow{\mathrm{esd}}\Delta\xrightarrow{\Hom_{\Delta}(-, \bullet)}\sSet
  \\ [n]\mapsto \mathrm{esd}[n]=[2n+1]\mapsto \Hom_{\Delta}(-, \mathrm{esd}[n])
\end{align*}
Notice by Yoneda's Lemma, 

\begin{align*}
    \mathrm{esd}^*(X)(n)=X(\mathrm{esd}[n])
    \cong\Hom_{\sSet}(\Hom_{\Delta}(-,\mathrm{esd}[n]),  X)
     =\Hom_{\sSet}(\mathrm{esd}[n],  X)
\end{align*}
In other words, $\mathrm{esd}^*(X)\cong \Hom_{\sSet}(\mathrm{esd}[\bullet],  X)$.
\begin{definition}
    Define natural transformations $i:id_\Delta\Rightarrow  \mathrm{esd}$ and $i^{op}:\mathrm{op}\Rightarrow\mathrm{esd}$ by 
    \begin{align*}
       i_n, i_n^{op}: [n]\to \mathrm{esd}[n]=\{\widehat n<...<\widehat{1}<\widehat{0}<0<1<...<n\}
    \end{align*}where $i_n(k)=k$ and $i_n^{op}(k)=\hat{k}$.
   
    Thus, for any simplicial set $X\in \sSet$, the above natural transformations induce simplicial maps
    \[
    \mathrm{src}_X:\mathrm{esd}^*(X)\to X,     \ \mathrm{tgt}_X:\mathrm{esd}^*(X)\to X^{op}
    \]
via the horizontal composition of natural transformation
  \begin{center}
\begin{tikzcd}
{\Delta^{op}} &&&& {\Delta^{op}} && \Set
	\arrow[""{name=0, anchor=center, inner sep=0}, "{\mathrm{op}}"{description}, shift right, curve={height=30pt}, from=1-1, to=1-5]
	\arrow[""{name=1, anchor=center, inner sep=0}, "{id_{\Delta^{op}}}"{description}, shift left, curve={height=-30pt}, from=1-1, to=1-5]
	\arrow[""{name=2, anchor=center, inner sep=0}, "{\mathrm{esd}}"{description}, from=1-1, to=1-5]
	\arrow["X", from=1-5, to=1-7]
	\arrow[Rightarrow, from=2, to=0, "{i^{op}}"{description, xshift=1pt, right}, pos=0.5, shorten <=10pt, shorten >=10pt]
	\arrow[Rightarrow, from=2, to=1, "{i}"'{description, xshift=1pt,right}, pos=0.5, shorten <=10pt, shorten >=10pt]
\end{tikzcd}
  \end{center}
  Here the direction of $i$ and $i^{op}$ are reversed because $X$ is contravariant.
\end{definition}

Observe that the functor  $\mathrm{esd}^*$ preserves limits, as it is corepresented by a cosimplicial simplicial set $\mathrm{esd}[\bullet]$. Moreover, each of the $\mathrm{esd}[n]$ is a compact object of $\sSet$, hence the functor also preserves filtered colimits. Therefore, there exists a left adjoint of $\mathrm{esd}^*$, denoted $\mathrm{esd}^!$. By Yoneda's Lemma, for any $n$, $\mathrm{esd}^!(\Delta[n])\cong \mathrm{esd}[n]$. 

For $X=\Delta[n]$, there exists 
\[i_{\Delta[n]}:\Delta[n]\to \mathrm{esd}^!\Delta[n], \ i^{op}_{\Delta[n]}:\Delta[n]^{op}\to \mathrm{esd}^!\Delta[n]\]
induced by the morphism $i_n, i_n^{op}:[n]\to \mathrm{esd}([n])$. These are the morphisms in adjunction with
\[\mathrm{src}_{\Delta[n]}:\mathrm{esd}^*(\Delta[n])\to \Delta[n],     \ \mathrm{tgt}_{\Delta[n]}:\mathrm{esd}^*(\Delta[n])\to \Delta[n]^{op}
\]
Moreover, these maps are monomorphisms of simplicial sets.

Since for every simplicial set $X\in \sSet$, 
\begin{align*}
    X\cong \colim(\Delta_{/X}\to \Delta\xrightarrow{\Delta[\bullet]}\sSet)
\end{align*}
The maps, $i_{\Delta[n]}$ and $i^{op}_{\Delta[n]}$ extends to simplicial maps $i_X:X\to \mathrm{esd}^!X$ and  $i_X^{op}:X^{op}\to \mathrm{esd}^!X$ that are in adjunction with $\mathrm{src}_X$ and $\mathrm{tgt}_X$, respectively.

For some simplicial sets, the maps $i_X$ and $i_X^{op}$ are anodyne, i.e. monomorphism and weak equivalence (we suspect this to be true for all $X\in \sSet$).
\begin{lemma}
 \label{anodyneesd}
 For any $n$ and $0\leq k \leq n$,
 \begin{align*}
     i_{\Delta[n]}:\Delta[n]\to \mathrm{esd}^!\Delta[n],
     \ i_{\partial\Delta[n]}:\partial\Delta[n]\to \mathrm{esd}^!\partial\Delta[n],
     \ i_{\Lambda^n_k}:\Lambda^n_k\to \mathrm{esd}^!\Lambda^n_k,
 \end{align*} and their $\mathrm{op}$ counterparts are anodyne.
\end{lemma}
\begin{proof}
    It is obvious that all maps are monomorphisms. Observe that $ \mathrm{esd}^!\Delta[n]\cong \Delta[2n+1]$, which is weakly contractible. This implies the case of $\Delta[n]$.

    We prove the case of $\Lambda^n_k$ inductively. 
    
    For $n=1$, $|\mathrm{esd}^!\Lambda^1_k|\cong |\Delta[1]|\simeq |\Lambda^1_k|$. Assume that for any $0\leq k\leq n-1$, $\mathrm{esd}^!\Lambda^{n-1}_k$ is weakly contractible. Notice that the subspace $S$ of $\Lambda^n_k$ corresponding to the intersecting $(n-2)$-simplices of the $(n-1)$-simplices comes from gluing $\Lambda^{n-1}_k$ together. Therefore, $\mathrm{esd}^!\Lambda^{n}_k$ has a subspace $\mathrm{esd}^! S$ that comes from gluing $\mathrm{esd}^!\Lambda^{n-1}_k$ together. 
    
    Each of the $|\mathrm{esd}^!\Lambda^{n-1}_k|$ is a strictly contractible subspace in $|\mathrm{esd}^! S|$. Then, contract the $|\mathrm{esd}^!\Lambda^{n-1}_k|$ in $|\mathrm{esd}^! S|$ until there is only one. Since the intersection of $|\mathrm{esd}^!\Lambda^{n-1}_k|$ with the rest of the subspace is contractible and $|\mathrm{esd}^!\Lambda^{n-1}_k|$ is contractible, then $|\mathrm{esd}^! S|$ is contractible. Therefore, $|\mathrm{esd}^!\Lambda^{n}_k|\simeq |\mathrm{esd}^!\Lambda^{n}_k|/|\mathrm{esd}^! S|$. Notice that $|\mathrm{esd}^!\Lambda^{n-1}_k|/|\mathrm{esd}^! S|$ is a bouquet of contractible spaces, therefore, it is contractible. 

    Similarly for $\partial \Delta[n]$. If $n=1$, $|\mathrm{esd}^!\partial \Delta[1]|\cong |\mathrm{esd}^!\Delta[0]^{\sqcup2}|\simeq |\partial \Delta[0]|$.
    
    Assume that $|\mathrm{esd}^!\partial \Delta[n-1]|\simeq |\partial\Delta[n-1]|$. Since $|\mathrm{esd}^!\Lambda^{n}_k|$ is a contractible subspace of $|\mathrm{esd}^!\partial \Delta[n]|$, so $|\mathrm{esd}^!\partial \Delta[n]|/|\mathrm{esd}^!\Lambda^{n}_k|$. Since $|\mathrm{esd}^!\partial \Delta[n]|-|\mathrm{esd}^!\Lambda^{n}_k|$ is the interior of $|\mathrm{esd}^!\Delta[n-1]|$, and the boundary of $|\mathrm{esd}^!\Delta[n-1]|$ is $|\mathrm{esd}^!\partial \Delta[n-1]|$. Therefore, $|\mathrm{esd}^!\partial \Delta[n]|/|\mathrm{esd}^!\Lambda^{n}_k|\simeq */S^{n-1}\simeq S^n\simeq |\partial\Delta[n]|$.

    Their opposite counterpart follows from observing $|\Delta[n]^{op}|\cong |\Delta[n]|$, $|\partial\Delta[n]^{op}|\cong |\partial\Delta[n]|$, and $|\Lambda^{n,op}_k|\cong |\Lambda^{n}_k|$.
\end{proof}
Since $\Lambda^n_k\to \Delta[n]$ is anodyne, the above also implies that $\mathrm{esd}^!\Lambda^n_k\to \mathrm{esd}^!\Delta[n]$ is anodyne.
\begin{lemma}
    For Kan complex $K\in \sSet$, the maps $\mathrm{src}_K:\mathrm{esd}^*(K)\to K$ and $\mathrm{tgt}_K:\mathrm{esd}^*(K)\to K^{op}$ are both Kan fibration. In particular, $\mathrm{esd}^*(K)$ is a Kan complex.
\end{lemma}
\begin{proof}
    The function $ \Hom_{\sSet}(\Delta[n], \mathrm{esd}^*(K))\cong \Hom_{\sSet}(\mathrm{esd}[n],K)\to \Hom_{\sSet}(\Lambda^{n}_k, K)$ is induced by the anodyne map $\Lambda^{n}_k\to \Delta[n]\to\mathrm{esd}[n]$. Since $K$ is a Kan complex, the function is surjective, which implies that $\mathrm{src}_X$ is a Kan fibration. The same is true for $\mathrm{tgt}_X$.
\end{proof}
If $X$ is a Kan complex, then $\mathrm{esd}^*(X)^{\partial\Delta[n]}\to X^{\partial\Delta[n]}$ is a Kan fibration for any $n$. Therefore, it is surjective on the connected components. 

On the other hand, suppose for $f,g:{\partial\Delta[n]}\to \mathrm{esd}^*(X)$, there is $\mathrm{src}_X\circ f\simeq \mathrm{src}_X\circ g$, via homotopy $H: {\partial\Delta[n]}\times\Delta[1]\to X$.

By Lemma~\ref{anodyneesd} and the fact that the product of anodyne maps is also anodyne, we have that the map $\partial\Delta[n]\times \Delta[1]\to \mathrm{esd}^!(\partial\Delta[n])\times \mathrm{esd}^!(\Delta[1])$ is anodyne. Therefore, the map 
\[X^{\mathrm{esd}^!(\partial\Delta[n])\times \mathrm{esd}^!(\Delta[1]})\to X^{\partial \Delta[n]\times \Delta[1]}
\] is a trivial Kan fibration and thus is bijective on the connected components. Let 
\[\tilde{H}:{\mathrm{esd}^!(\partial\Delta[n])\times \mathrm{esd}^!(\Delta[1]})\to X\]
be a choice of a simplicial map in the connected component corresponding to that of $H$ in $X^{\partial \Delta[n]\times \Delta[1]}$. Precomposing $\tilde{H}$ with the canonical map $\mathrm{esd}^!(\partial\Delta[n]\times \Delta[1])\to\mathrm{esd}^!(\partial\Delta[n])\times \mathrm{esd}^!(\Delta[1])$ determines a map in $\Hom_{\sSet}(\mathrm{esd}^!(\partial\Delta[n]\times \Delta[1]), X)$, with the adjuncting map determining a homotopy between $f$ and $g$.

In other words, the function $[\partial \Delta[n], \mathrm{esd}^*(X)]\to [\partial \Delta[n], X]$ induced by $\mathrm{src}_X$ is bijective. Therefore, we have the following.

\begin{proposition}
  \label{esdopposite} For any Kan complex $K\in \sSet$, both $\rm{src}_K$ and $\mathrm{tgt}_K$ are homotopy equivalences. Moreover,  for any simplicial set $X\in \sSet$, there is a zigzag of weak homotopy equivalences
  \[X\simeq \mathrm{esd}^*(X)\simeq X^{op}\]
\end{proposition}
\begin{proof}
For any Kan complex $K$, the above argument implies that $\mathrm{src}_K$ and $\mathrm{tgt}_K$ are weak homotopy equivalences. Moreover, let $X\to K$ be a Kan replacement of $X$. Then there is
\begin{center}
   \begin{tikzcd}
X \arrow[r, "\simeq"] & K & \mathrm{esd}^*K \arrow[l, "\mathrm{src}_K"'] \arrow[r, "\mathrm{tgt}_K"] & K^{op} & X^{op} \arrow[l, "\simeq"']
\end{tikzcd}
\end{center}
\end{proof}
\section{Proof of The Main Theorem}\label{section4}
The goal of this section is to prove the main result of this paper, Theorem~\ref{main}, which says that given any fibrant-cofibrant object $A$ in a simplicial model category $\scr{M}$, the space $B\Aut(A)$ is weakly equivalent to $\hat{A}$.\textit{
}

In order to prove this theorem, we construct a zig zag of weak equivalences. 
    \begin{equation}
    \label{zigzag}
    \begin{tikzcd}
B\Aut(A) \arrow[rd,"\text{Prop. } \ref{first}"] &                                          & {{|\scr{N}^{\mathrm{}}\scr{C}_{A,w}|}\simeq {|\scr{N}^{\mathrm{}}\scr{C}^{op}_{A,w}|}} \arrow[ld,"\text{Prop. }\ref{second}"] \arrow[rd, "\text{Prop. } \ref{third}"] &                                              & {\|\fgt\hat{A}_\bullet\|\simeq \hat{A}} \\
                    & {{|\scr{N}^{\mathrm{hc}}\scr{C}_{A,w}|}} &                                                                                                              & {|\hat{A}'_\bullet|\simeq \|\fgt\hat{A}'_\bullet\|} \arrow[ru,"\text{Prop. } \ref{fourth}"] &                                
\end{tikzcd}
    \end{equation}
  
\subsection{The Weak Homotopy Equivalence \protect\boldmath$B\Aut(A)\to |\scr{N}^{\mathrm{hc}}\scr{C}_{A,w}|$ }\label{section41}

This section is dedicated to the first weak homotopy equivalence of Diagram~(\ref{zigzag}). Recall the definition of a simplicial monoid.
\begin{definition}
    A simplicial monoid is a simplicial set $G$, equipped with a simplicial map $\varphi:G\times G\to G$ and $e: \Delta[0]\to G$, such that the following diagrams commute:
    \begin{center}
       \begin{tikzcd}
G\times G\times G \arrow[r, "id\times\varphi"] \arrow[d, "\varphi\times id"'] & G\times G \arrow[d, "\varphi"] & G \arrow[d, "e \times id"'] \arrow[r, "id \times e"] \arrow[rd, "id"] & G\times G \arrow[d, "\varphi"] \\
G\times G \arrow[r, "\varphi"]                                                & G                              & G\times G \arrow[r, "\varphi"]                                        & G                             
\end{tikzcd}
    \end{center}
\end{definition}
Similar to the construction of the classifying space of a topological group, there exists the concept of the \textit{classifying space} of a simplicial monoid. Given a small simplicial category $\scr{C}$, let \[
\scr{M}\mathrm{or}(\scr{C}):=\coprod_{A,B\in \scr{O}\mathrm{bj}(\scr{C})}\Map_{\scr{C}}(A,B)\]
the \textit{enriched nerve} $\scr{N}^\Delta\scr{C}$ is the bisimplicial set defined as follows:
\begin{align*}
    \scr{N}^\Delta\scr{C}(0):=\scr{O}\mathrm{bj}(\scr{C})
    \end{align*}Here $\scr{O}\mathrm{bj}(\scr{C})$ is viewed as a discrete simplicial set, and 
    \begin{align*}
 \scr{N}^\Delta\scr{C}(n):=\scr{M}\mathrm{or}(\scr{C})\underset{\scr{O}\mathrm{bj(\scr{C})}}{\times}\scr{M}\mathrm{or}(\scr{C})...\scr{M}\mathrm{or}(\scr{C})\underset{\scr{O}\mathrm{bj(\scr{C})}}{\times}\scr{M}\mathrm{or}(\scr{C}) \text{\ ($n$ times)}
\end{align*}
The face $d_i$ is the composition of the $i$th and $(i+1)$th  morphisms, and the degeneracy $s_i$ is the insertion of identity at the $i$th position. 

\begin{definition}
    Let $G$ be a simplicial monoid, and let $\scr{C}_G$ be the simplicial category with one object $\mathbb{I}$ and $\Map(\mathbb{I},\mathbb{I}):=G$. The classifying space of $G$ is defined to be the diagonal of the enriched nerve of $\scr{C}_G$; that is, $BG:=\mathrm{diag}\scr{N}^{\Delta}(\scr{C}_G)$.
\end{definition}
The simplicial monoid that we concern is $\Aut(A)$.
\begin{definition}
    Let $\scr{M}$ be a simplicial model category and let $A\in \scr{M}$ be a fibrant-cofibrant object. We define $\Aut(A)$ to be the connected components of $ \Map_\scr{M}(A,A)$ that are weak equivalences. Since weak equivalences are closed under composition, this is a simplicial monoid. 
\end{definition}
Of course, since $\Aut(A)$ comes from a simplicial model category, it is natural to consider the homotopy coherent nerve instead of the enriched nerve. In fact, we have the following result. 
\begin{lemma}
\label{lemma42} Let $G$ be a simplicial monoid such that the underlying simplicial set is a Kan complex. Then there exists a weak homotopy equivalence $BG\to \scr{N}^{\mathrm{hc}}\scr{C}_G$.
\end{lemma}
\begin{proof}
    It was shown in \cite[Corollary 1.8]{arakawa2025classificationdiagramssimplicialcategories} that given a locally Kan category $\scr{C}$, there is a weak homotopy equivalence $\mathrm{diag}\scr{N}^{\Delta}\scr{C}\to \scr{N}^{\mathrm{hc}}\scr{C}$. Applying this to $\scr{C}=\scr{C}_G$, we get the desired result.
\end{proof}

\begin{definition}
    Given any simplicial model category $\scr{M}$ and any fibrant-cofibrant object $A\in \scr{M}$, let $\scr{C}_A$ be the category (viewed as the full simplicial subcategory of $\scr{M}$) described in Definition~\ref{CAcategory}. Define $\scr{C}_{A,w}$ to be the subcategory of $\scr{C}_A$ whose objects are the same, with the mapping space the connected components of the mapping space of $\scr{C}_A$ that are weak equivalences.
\end{definition}
First of all, since objects in $\scr{C}_A$ are fibrant-cofibrant in $\scr{M}$, the category $\scr{C}_A$ and the subcategory $\scr{C}_{A,w}$ are locally Kan categories. Moreover, by the well known result of  \cite[II.1, Theorem 1.10 (Whitehead)]{goerss-jardine}, we have that weak equivalences between fibrant-cofibrant objects are homotopy equivalences. Therefore, due to Lemma~\ref{Kancomplex}, the homotopy coherent nerve of $\scr{C}_{A,w}$, $\scr{N}^{\mathrm{hc}}\scr{C}_{A,w}$, is a Kan complex. 

Define a simplicial map $in_A:B\Aut(A)\to \scr{N}^{\mathrm{hc}}\scr{C}_{A,w}$ to be the composition of $B\Aut(A)\to \scr{N}^{\mathrm{hc}}\scr{C}_{\Aut(A)}$ with the homotopy coherent nerve of the functor $\scr{C}_{\Aut(A)}\to \scr{C}_{A,w}$ that takes $\mathbb{I}$ to $A$, and $\Aut(A)$ to $\Map_{ \scr{C}_{A,w}}(A,A)=\Aut(A)$.

\begin{proposition}\label{first}
    Let $\scr{M}$ be a simplicial model category and let $A\in \scr{M}$ be fibrant-cofibrant, and let  $\scr{C}_{A,w}$ be the category above. Then $in_A:B\Aut(A)\to  \scr{N}^{\mathrm{hc}}\scr{C}_{A,w}$ is a weak homotopy equivalence.
\end{proposition}
\begin{proof}
          Recall that the adjoint pair $\mathfrak{C}\dashv \scr{N}^{\mathrm{hc}}:\sSet\to \sCat$ is a Quillen equivalence between the Dwyer-Kan-Bergner model structure of $\sCat$ and the Joyal model structure of $\sSet$. So if there is a simplicial functor $F:\scr{C}\to \scr{D}$ between locally Kan categories such that:
        \begin{itemize}
            \item $\mathrm{Ho}(F):\mathrm{Ho}(\scr{C})\to \mathrm{Ho}(\scr{D})$ is essentially surjective
            \item For any $x,y\in \scr{C}$, $F_{x,y}:\Map_{\scr{C}}(x,y)\to \Map_{\scr{D}}(F(x),F(y))$ is a weak homotopy equivalence.

        \end{itemize}
    Then $\scr{N}^{\mathrm{hc}}(F)$ is a categorical equivalence of $\sSet$, which implies weak homotopy equivalence. We see that the simplicial functor $\scr{C}_{\Aut(A)}\to \scr{C}_{A,w}$ satisfies the criteria, so $\scr{N}^\mathrm{hc}\scr{C}_{\Aut(A)}\to \scr{N}^\mathrm{hc}\scr{C}_{A,w}$  is a weak homotopy equivalence. Moreover, by Lemma~\ref{lemma42}, we have $B\Aut(A)\to \scr{N}^\mathrm{hc}\scr{C}_{\Aut(A)}$ is a weak homotopy equivalence. Therefore, their composition $in _A$ is a weak homotopy equivalence.
\end{proof}

\subsection{The Weak Homotopy Equivalence \protect\boldmath$\scr{N}\scr{C}_{A,w}\to \scr{N}^{\mathrm{hc}}\scr{C}_{A,w}$}
\label{section42}

In this section, we will show the second weak homotopy equivalence in Diagram~\ref{zigzag}. We first define the map.
\begin{definition}
    For all $n\geq 0$, there is an inclusion $[n]\hookrightarrow\mathrm{Path}[n]$ that takes objects $i\in [n]$ to $i\in \mathrm{Path}[n]$ and takes morphism $(i)\to (i+j)$ to $[i<i+1<...<i+j]\in \Map(i,i+j)$.
    
      For a small simplicial category $\scr{C}$, there exists a natural simplicial map
\[\varphi_{\scr{C}}:\scr{N}\scr{C}\to \scr{N}^{\mathrm{hc}}\scr{C}
\]that is induced by inclusion $[n]\hookrightarrow\mathrm{Path}[n]$. 
\end{definition}

The map in the title is the map $\varphi_A:=\varphi_{\scr{C}_{A,w}}:\scr{N}\scr{C}_{A,w}\to \scr{N}^{\mathrm{hc}}\scr{C}_{A,w}$ for the small simplicial category $\scr{C}_{A,w}$. The map $\varphi_{\scr{C}}$ is generally not a weak homotopy equivalence for an arbitrary small simplicial category.  Therefore, we will need several results to show that $\varphi_A$ is a weak homotopy equivalence.

First, we recall a (slightly weaker version of) proposition in \cite{lurieHA}. 
\begin{proposition}
  \label{ordinarynerve}  \cite[Proposition 1.3.4.7.]{lurieHA} Let $\scr{C}$ be a locally Kan category with a class of morphisms $\scr{W}$ of $\scr{C}$ such that 
    \begin{itemize}
        \item $\scr{W}$ contains all the isomorphisms in $\scr{C}$.
        \item $\scr{W}$ has $2$-out-of-$3$ property.
        \item There is a tensoring of $\Delta[1]$ on $\scr{C}$, i.e., there is a functor
\begin{align*}
    \Delta[1]\otimes: \scr{C}\to \scr{C}
\end{align*} and for any $A,B\in \scr{C}$, an isomorphism of simplicial sets 
\begin{align*}
    h_{A,B}:\Map_\scr{C}(\Delta[1]\otimes A,B)\to \Map_{\sSet}(\Delta[1], \Map_\scr{C}(A,B))
\end{align*}
\item For any $A\in \scr{C}$, the canonical map $\Delta[1]\otimes A\to A$ is in $\scr{W}$.

    \end{itemize}
Then the natural simplicial map $\varphi:\scr{N}\scr{C}\to \scr{N}^{\mathrm{hc}}\scr{C}$ that determines a weak equivalence of marked simplicial sets $(\scr{N}\scr{C},\scr{W})\to (\scr{N}^{\mathrm{hc}}\scr{C}, \scr{W})$.
\end{proposition}
Let $\scr{M}$ be our simplicial model category and let $\scr{C}_{A,w}$ be the same as in Section \ref{section41}. We wish to use the above proposition to prove Proposition~\ref{second}, which is the main result of this section. The issue for $\scr{C}_{A,w}$ is that it might not admit a tensoring of $\Delta[1]$. We circumvent this issue as follows. Defintion~\ref{CAcategory} characterizes $\scr{C}_A$ through several criteria, one of which requires $\scr{C}_A$ to be closed under some cylinder object functor $Z:\scr{C}_A\to \scr{C}_A$. This functor $Z$ is arbitrary. On the other hand, the simplicial model category $\scr{M}$ has a tensoring of any simplicial set by definition, and thus, the functor $\Delta[1]\otimes-:\scr{M}\to \scr{M}$ exists. It is a standard fact that $\Delta[1]\otimes -$ is a good cylinder object functor. Therefore, we may first choose the functor $Z$ that defines our category $\scr{C}_A$ to be $\Delta[1]\otimes-$. In order to distinguish this category we obtained from $\Delta[1]\otimes-$ from others, we denote the category as $\scr{C}_A'$, and we denote $\scr{C}_{A,w}'$ to be the subcategory of weak equivalences of $\scr{C}_A$, as in Section~\ref{section41}.
    
\begin{lemma}
   \label{BCzaif} For any $B,C\in\scr{C}_{A,w}'$, there is an isomorphism 
    \begin{align*}
        h_{B,C}:\Map_{\scr{C}_{A,w}'}(\Delta[1]\otimes B,C)\to \Map_{\sSet}(\Delta[1], \Map_{\scr{C}_{A,w}'}(B,C))
    \end{align*} 
\end{lemma}
\begin{proof}
    First, since $\scr{M}$ is a simplicial model category, and $\scr{C}_A'$ is by definition a full subcategory of $\scr{M}$, there are isomorphisms
    \begin{align*}
           h_{B,C}':\Map_{\scr{C}_{A}'}(\Delta[1]\otimes B,C)\to \Map_{\sSet}(\Delta[1], \Map_{\scr{C}_{A}'}(B,C))
    \end{align*} induced from the tensoring of $\Delta[1]$ in $\scr{M}$.
   We define $h_{B,C}$ to be the restriction of $h_{B,C}'$ on the subsimplicial set
    \begin{align*}
        \Map_{\scr{C}_{A,w}'}(\Delta[1]\otimes B,C)\subseteq \Map_{\scr{C}_{A}'}(\Delta[1]\otimes B,C)
    \end{align*}which is by definition the connected components of $\Map_{\scr{C}_{A}'}(\Delta[1]\otimes B,C)$ of weak equivalences. This is then an isomorphism onto the image.
    On the other hand, observe that the simplicial set
    \begin{align*}
        \Map_{\sSet}(\Delta[1], \Map_{\scr{C}_{A,w}'}(B,C))\hookrightarrow \Map_{\sSet}(\Delta[1], \Map_{\scr{C}_{A}'}(B,C))
    \end{align*} is precisely the image of the $h_{B,C}$. The lemma then follows.
\end{proof}

Therefore, combining Proposition~\ref{ordinarynerve} with Lemma~\ref{BCzaif}, we have the following.
\begin{proposition}
   \label{anchor} The natural simplicial map $\varphi:\scr{N}\scr{C}_{A,w}'\to \scr{N}^{\mathrm{hc}}\scr{C}_{A,w}'$ is a weak homotopy equivalence.
\end{proposition}

Of course, Proposition~\ref{anchor} only shows the case of $Z=\Delta[1]\otimes-$. The next proposition generalize the result for an arbitrary $Z$.
\begin{proposition}
       \label{second}
    For any cylinder object functor $Z$, let $\scr{C}_{A,w}$ be the category described in Section~\ref{section41} associated to $Z$. Then, the natural simplicial map $\varphi:\scr{N}\scr{C}_{A,w}\to \scr{N}^{\mathrm{hc}}\scr{C}_{A,w}$ is a weak homotopy equivalence.
\end{proposition}
\begin{proof}

For any $D\in \Cat$, given $F, F\in \Fun(D,\scr{C}_{A,w})$, let $\mathfrak{w}$ be the equivalence relation such that $F\mathfrak{w}F$ if and only if there exists a zigzag of objectwise weak equivalences between $F$ and $F'$. 

 The main result of \cite[Section 4]{tsopmene2024classificationhomogeneousfunctorsmanifold}  is a series of results \cite[Proposition 4.8, Proposition 4.15]{tsopmene2024classificationhomogeneousfunctorsmanifold} that determine a bijection.
\[{\scr{F}_{kA}(\scr{O}(M)^{op}, \scr{M})/\frak{w}}\to \Fun(U(\scr{T}^M)^{op},\scr{C}_{A,w})/\mathfrak{w}\]
here $ {\scr{F}_{kA}(\scr{O}(M)^{op}, \scr{M})/\frak{w}}$ is the equivalent classes of a type of homogeneous functors that will be introduced in Section~\ref{section46}, and $U(\scr{T}^M)$ is a category associated to the triangulation $\scr{T}^M$ of some manifold $M$, see \cite[Section 4]{tsopmene2024classificationhomogeneousfunctorsmanifold} for details.

Notice that for any small category $D$, $F\mathfrak{w}F'\in \Fun(D,\scr{C}_{A,w})$ is equivalent to the fact that $\scr{N}F, \scr{N}F'\in \Fun(\scr{N}D,\scr{N}\scr{C}_{A,w})$ live in the same connected component. Therefore, shows that there exists a bijection, 
\begin{align*}
    {\scr{F}_{kA}(\scr{O}(M)^{op}, \scr{M})/\frak{w}}\to \pi_0\Fun(\scr{N}U(\scr{T}^M)^{op},\scr{N}\scr{C}_{A,w})
\end{align*}
Let $\scr{M}_{A,w}$ denote the subcategory of $\scr{M}$ whose objects are fibrant-cofibrant and are homotopy equivalent to $A$, and the mapping space are the connected components of weak equivalences. We have the following diagram
\begin{center}
    \begin{tikzcd}
{\pi_0\Fun(\scr{N}U(\scr{T}^M)^{op},\scr{N}^{\mathrm{hc}}\scr{M}_{A,w})} & {\pi_0\Fun(\scr{N}U(\scr{T}^M)^{op},\scr{N}\scr{C}_{A,w}')} \arrow[l]                                      \\
{\pi_0\Fun(\scr{N}U(\scr{T}^M)^{op},\scr{N}\scr{C}_{A,w})} \arrow[u]        & {\scr{F}_{kA}(\scr{O}(M)^{op}, \scr{M})/\frak{w}} \arrow[l, "\cong"] \arrow[u, "\cong"]
\end{tikzcd}
\end{center} The top and left arrows are induced from the inclusions $\scr{C}_{A,w}$ and $ \scr{C}_{A,w}'\hookrightarrow\scr{M}_{A,w}$. 
The diagram above commutes, as $\scr{}\scr{C}_{A,w}$ and $\scr{}\scr{C}_{A,w}'$ have objects of the same homotopy type. Moreover, the top arrow is an isomorphism by Proposition~\ref{anchor}, so 
\begin{align*}
    {\pi_0\Fun(\scr{N}U(\scr{T}^M)^{op},\scr{N}\scr{C}_{A,w})} \to {\pi_0\Fun(\scr{N}U(\scr{T}^M)^{op},\scr{N}^{\mathrm{hc}}\scr{M}_{A,w})}
\end{align*} is an isomorphism. 

This function is induced by the natural inclusion  $\scr{N}\scr{C}_{A,w}\to \scr{N}^\mathrm{hc}\scr{C}_{A,w}\xrightarrow{\sim} \scr{N}^\mathrm{hc}\scr{M}_{A,w}$, and we factor it as \[\scr{N}\scr{C}_{A,w}\stackrel{\sim}{\rightarrowtail{}}K\twoheadrightarrow \scr{N}^\mathrm{hc}\scr{M}_{A,w}\]

We have that if $f:X\to Y$ is a weak equivalence of simplicial set, then for any simplicial set $Z\in \sSet$, there is a bijection
\[
\pi_0\Fun(Z,X)\to \pi_0\Fun(Z,Y)
\]
Then, the function
\begin{align*}
       {\pi_0\Fun(\scr{N}U(\scr{T}^M)^{op},\scr{N}\scr{C}_{A,w})} \cong \pi_0\Fun(\scr{N}U(\scr{T}^M)^{op},K)\to {\pi_0\Fun(\scr{N}U(\scr{T}^M)^{op},\scr{N}^{\mathrm{hc}}\scr{M}_{A,w})}
\end{align*} is a bijection. 

For any $n\geq 1$, we are able to find a triangulation $\scr{T}^{S^{n-1}}$ of the $n-1$-sphere $S^{n-1}$ such that $U(\scr{T}^{S^{n-1}})\cong \partial\tilde\Delta^{n}$. Since both $K$ and $\scr{N}^{\mathrm{hc}}\scr{M}_{A,w}$ are Kan complexes, there is an isomorphism,
\begin{multline*}
\pi_0\Fun(\scr{N}(\partial\tilde\Delta^{n})^{op},K)\cong [|(\partial\tilde\Delta^{n})^{op}|, |K|]\cong [S^{n-1}, |K|]
\\     \xrightarrow{\cong}{\pi_0\Fun(\scr{N}(\partial\tilde\Delta^{n})^{op},\scr{N}^{\mathrm{hc}}\scr{M}_{A,w})}\cong [S^{n-1}, |\scr{N}^{\mathrm{hc}}\scr{M}_{A,w}|]
\end{multline*}

Therefore, we have a weak equivalence $K\simeq \scr{N}^{\mathrm{hc}}\scr{M}_{A,w}$. Thus, by the two-out-of-three property of weak equivalences, we have that $\scr{N}\scr{C}_{A,w}\simeq \scr{N}^\mathrm{hc}\scr{C}_{A,w}$.
\end{proof}
\subsection{The Weak Homotopy Equivalence \protect\boldmath$\scr{N}\scr{C}_{A,w}\simeq\hat{A}_\bullet'$}
\label{section43}

In this section, we will show that there is a weak homotopy equivalence between $\scr{N}\scr{C}_{A,w}$ and $\hat{A}'_\bullet$. This weak homotopy equivalence can be described as a special case of a more general result regarding any simplicial set. 

In \cite[Chapter III, 4.]{goerss-jardine}, the author defines a functor $\mathrm{Ex}: \sSet\to \sSet$ as follows: given a simplicial set $X\in \sSet$, define 
    \begin{align*}
        \mathrm{Ex}(X)(n):=\Hom_{\sSet}(\scr{N}\tilde\Delta^n, X),
    \end{align*}
    where the category $\tilde{\Delta}^n$ is the same as the one we introduced in Section~\ref{section2}. Since $\tilde\Delta^\bullet$ is a cosimplicial object in the category of small categories, the collection $\{\mathrm{Ex}(X)(n)\}_{n\geq 0}$ determines a simplicial set. Moreover, this determines a functor $\mathrm{Ex}:\sSet\to \sSet$.

    This functor admits a left adjoint that will be denoted $\mathrm{sd}:\sSet\to \sSet$ and will be referred to as the\textit{ barycentric subdivision functor}. For the cosimplicial simplicial set of standard simplicial sets $\Delta[\bullet]$, let the cosimplicial simplicial sets $\mathrm{sd}\Delta[\bullet]$ be obtained as follow
    \begin{align*}
        \mathrm{sd}\Delta[\bullet]:\Delta\xrightarrow{\Delta[\bullet]} \sSet\xrightarrow{\mathrm{sd}}\sSet.
    \end{align*}For any $X\in \sSet$, we have
    \begin{align*}
        \Hom_{\sSet}(\mathrm{sd}\Delta[\bullet],X)\cong \Hom_{\sSet}(\Delta[\bullet],\mathrm{Ex}(X))
\cong\mathrm{Ex}( X)(\bullet)
    \end{align*} where the first isomorphism comes from the adjunction and the second isomorphism comes from Yoneda's Lemma. Therefore, $\mathrm{sd}\Delta[\bullet]\cong \scr{N}\tilde\Delta^\bullet$ (even though the notation are similar, the construction above is irrelevant to the construction in Section~\ref{esd}).
    
  Observe that the category induced by the totally ordered set $[n]$ has nerve $\scr{N}[n]\cong\Delta[n]$. There is also a natural transformation $\tilde\Delta^\bullet\to [\bullet]$ defined by $\{k_0<k_1<...k_m\}\mapsto (k_m)$. This is known to be the last vertex map. Moreover, since $X\cong\Hom_{\sSet}(\scr{N}[\bullet], X)$, this induces a simplicial map $\eta_X:X\to\mathrm{Ex}(X)$.

\begin{proposition}
\label{lastvertex}    \cite[Theorem 4.6]{goerss-jardine} For any $X\in\sSet$, the map $\eta_X:X\to\mathrm{Ex}(X)$ is a weak homotopy equivalence.
\end{proposition}
We wish to use Proposition~\ref{lastvertex} to study the simplicial set $\hat{A}'_\bullet$ introduced in Section~\ref{section2}.
Observe the simplicial set $\hat{A}'_\bullet$ is nothing but the simplicial set  $\Hom_{\sCat}((\tilde\Delta^\bullet)^{op}, \scr{C}_{A,w})$. Notice that the set of contravariant functors in a category is naturally isomorphic to the set of covariant functors in the opposite category, $\hat{A}_\bullet'\cong \Hom_{\sCat}(\tilde\Delta^\bullet, \scr{C}_{A,w}^{op})$. Since the ordinary nerve functor $\scr{N}$ is fully faithful,

    \begin{equation}
    \label{AhatprimeAsEX}
        \hat{A}_\bullet'\cong \Hom_{\sCat}(\tilde\Delta^\bullet, \scr{C}_{A,w}^{op})\cong  \Hom_{\sCat}(\scr{N}\tilde\Delta^\bullet, \scr{N}\scr{C}_{A,w}^{op})\cong\mathrm{Ex}(\scr{N}\scr{C}^{op}_{A,w})
    \end{equation}
 We then have a corollary of Proposition~\ref{lastvertex}.
\begin{corollary}
     \label{ExA}   There is a weak equivalence $\eta_{\scr{N}\scr{C}^{op}_{A,w}}:\scr{N}\scr{C}^{op}_{A,w}\to \mathrm{Ex}(\scr{N}\scr{C}^{op}_{A,w})\cong\hat{A}'$.
\end{corollary}
\begin{proposition}\label{third}
    There is a weak homotopy equivalence $\scr{N}\scr{C}_{A,w}\simeq \hat{A}_\bullet'$
\end{proposition}
\begin{proof}
    This follows from Corollary~\ref{ExA} and Proposition~\ref{esdopposite}.
\end{proof}

\subsection{The Homotopy Equivalence \protect\boldmath$\|\fgt\hat{A}_\bullet\|\to \|\fgt\hat{A}'_\bullet\|$}
\label{section44}
The idea to the proof of the main result is to construct a zigzag of weak equivalences between $\hat{A}_\bullet$ and $B \Aut(A)$, and so far we have constructed such a zigzag between $\hat{A}'_\bullet$ and $B \Aut(A)$. Now we will construct a weak equivalence $\hat{A}'_\bullet\to \hat{A}_\bullet$. Recall that we mentioned in Section~\ref{section2} that there is not a simple degeneracy structure on $\hat A_\bullet$. This makes the study of $\hat A_\bullet$ difficult to take place in $\sSet$. Therefore, we use Proposition~\ref{fgtsimeq} from Section~\ref{section31}, which states that $\|\fgt (X)\|\simeq |X|$, allowing us to consider the semisimplicial set $\fgt \hat A_\bullet$ instead. 

Notice that even though there does not exist a very natural simplicial map between $\hat{A}_\bullet'$ and $\hat{A}_\bullet$, there is a natural semisimplicial map
\[\Phi:\fgt\hat{A}_\bullet\to \fgt\hat{A}'_\bullet\]
induced by the inclusion of $\hat{A}_n$.

In \cite{tsopmene2024classificationhomogeneousfunctorsmanifold}, the authors introduced a functor $R_n:\scr{C}_{A,w}^{(\tilde{\Delta}^{n})^{op}}=\fgt \hat{A}'_n\to \scr{C}_{A,w}^{(\tilde{\Delta}^{n})^{op}}$ such that for any $F\in \scr{C}_{A,w}^{(\tilde{\Delta}^{n})^{op}}$, $R_n F$ is fibrant in the injective model structure of $\scr{M}^{(\tilde{\Delta}^{n})^{op}}$. In other words, $R_nF\in \hat{A}_n$. We define 
\[
    R:\fgt\hat{A}'_\bullet\to \fgt\hat{A}_\bullet, \  R(F):=R_nF
\]
This is a semisimplicial map because for any $d^i:{(\tilde{\Delta}^{n-1})^{op}}\to {(\tilde{\Delta}^{n})^{op}}$, $R_{n-1}(F\circ d^i)=(R_nF)\circ d^i$.

Since for any $F\in\hat{A}_n$, $R_nF=F$, then $R\circ \Phi=id_{\fgt \hat A_\bullet}$, i.e. $\fgt\hat{A}$ is a retract of $\fgt\hat{A}'$.

We wish to show that $R$ induces a weak homotopy equivalence between fat realizations. Our approach is to show that for every $n$-simplex of $\fgt \hat{A}'_\bullet$, there exists a cylinder $\Delta^n\times \Delta^1$ embedded in $\|\fgt \hat{A}'_\bullet\|$ connecting the $n$-simplex to the subspace $\|\fgt \hat{A}_\bullet\|$, such that the collection of cylinders determines a weak equivalence between $\|\fgt\hat{A}'_\bullet\|$ and $\|\fgt\hat{A}_\bullet\|$. 

In $\sSet$, a map $\Delta[n]\times \Delta[1]\to X$ determines a map $\Delta^n\times \Delta^1\to |X|$. However, for semisimplicial sets, an $n$-simplex is determined by a semisimplicial map $\mathrm{Core}(\Delta[n])\to X$ due to Yoneda's Lemma and the fact that $\mathrm{Core}(\Delta[n])\cong \Hom_{\Delta_\mathrm{inj}}(-, [n])$. However, $\mathrm{Core}(\Delta[n])\times\mathrm{Core}(\Delta[1])$ has no simplex of dimension $\geq 2$. Therefore, simply considering $\mathrm{Core}(\Delta[n])\times\mathrm{Core}(\Delta[1])$ is insufficient. 

We avoid the above issue by noticing that $\Delta[n]\times \Delta[1]$ is n.d.c. Therefore, $\free \mathrm{Core}(\Delta[n]\times\Delta[1])\cong \Delta[n]\times\Delta[1]$. Moreover, both of our semisimplicial sets come from simplicial sets, and we have the following property from Lemma~\ref{phiX}
    \begin{equation}
    \label{cylinderX}
       \Hom_{\semisSet}(\mathrm{Core}(\Delta[n]\times\Delta[1]), \fgt X)\cong \Hom_{\sSet} (\Delta[n]\times\Delta[1], X)
\end{equation}
Therefore, to construct a cylinder in $\|\fgt X\|$, we may consider the semisimplicial maps \[\mathrm{Core}(\Delta[n]\times\Delta[1])\to \fgt X\]
\begin{lemma}
\label{coreandsd}   There is an isomorphism\[
          \Hom_{\semisSet}(\mathrm{Core}(\Delta[n]\times\Delta[1]),\fgt\hat{A}'_\bullet)\cong  \Hom_{\sSet}(\mathrm{sd}(\Delta[n]\times\Delta[1]),\scr{NC}_{A,w}^{op})
   \] 
  Moreover, given a semisimplicial map $H:\mathrm{Core}(\Delta[n]\times\Delta[1])\to \fgt\hat{A}_\bullet'$, denote $\tilde H:\mathrm{sd}(\Delta[n]\times \Delta[1])\to \scr{N}\scr{C}_{A,w}^{op}$ its corresponding simplicial map. For $i=0,1$, $H([n], i)=\tilde{H}|_{\mathrm{sd}(\Delta[n]\times i)}$. Here, $([n],i) \in \mathrm{Core}(\Delta[n]\times\Delta[1])_{n}$ is the $n$-simplex at the vertex $i$ of $ \Delta[1]$.
\end{lemma}
\begin{proof}
    There is a natural isomorphism
    \begin{align*}
        \Hom_{\semisSet}(\mathrm{Core}(\Delta[n]\times\Delta[1]),\fgt\hat{A}'_\bullet)\cong      \Hom_{\sSet}(\Delta[n]\times\Delta[1],\hat{A}_\bullet')
        \\ \cong \Hom_{\sSet}(\Delta[n]\times\Delta[1],\mathrm{Ex}(\scr{NC}_{A,w}^{op}))
        \\ \cong  \Hom_{\sSet}(\mathrm{sd}(\Delta[n]\times\Delta[1]),\scr{NC}_{A,w}^{op})
    \end{align*}
   The first isomorphism comes from Equation~\ref{cylinderX}, the second one comes from Equation~\ref{AhatprimeAsEX}, and the last one comes from the adjunction $\mathrm{sd}\dashv \mathrm{Ex}$. 

   We can use the same argument to show that \[ \Hom_{\semisSet}(\mathrm{Core}(\Delta[n]),\fgt\hat{A}'_\bullet)\cong  \Hom_{\sSet}(\mathrm{sd}(\Delta[n]),\scr{NC}_{A,w}^{op})\] Observe the following diagram. 
   \begin{center}
       \begin{tikzcd}
{\Hom_{\semisSet}(\mathrm{Core}(\Delta[n]\times\Delta[1]),\fgt\hat{A}'_\bullet)} \arrow[d, "{H\mapsto H|_{([n],i)}}"'] \arrow[rr, "H\mapsto \tilde{H}"] &  & {\Hom_{\sSet}(\mathrm{sd}(\Delta[n]\times\Delta[1]),\scr{NC}_{A,w}^{op})} \arrow[d, "{\tilde{H}\mapsto \tilde{H}|_{([n],i)}}"] \\
{\Hom_{\semisSet}( \mathrm{Core}(\Delta [n]),\fgt\hat{A}'_\bullet)} \arrow[rr, "F\mapsto\tilde{F}"]                                                     &  & {\Hom_{\sSet}(\mathrm{sd}(\Delta[n]),\scr{NC}_{A,w}^{op})}                                                                    
\end{tikzcd}
   \end{center}where $\tilde{F}$ is the simplicial map $\mathrm{sd}(\Delta[n])\to\scr{NC}_{A,w}^{op}$ corresponding to $F:\mathrm{Core}(\Delta [n])\to \fgt\hat{A}'_\bullet$.
   
It is not hard to see that this diagram commutes and shows our second claim.
\end{proof}
The construction of the cylinders would rely on the following property of $R$. For any $F\in \hat{A}'_\bullet$, there is a natural weak equivalence \[
\eta_F: F\to R(F)\]
such that $\eta_F|_{d^i(\tilde\Delta^{n-1,op})}=\eta_{F\circ d^i}$ (see \cite[3.2]{tsopmene2024classificationhomogeneousfunctorsmanifold} for details). Equivalently, there is a simplicial map $\eta_F:\scr{N}(\tilde\Delta^{n})^{op}\times \Delta[1]\to \scr{N}\scr{C}_{A,w}$.

\begin{lemma}
   \label{cylinderA} For any $F \in \fgt\hat{A}_n'$, there is a map $\mathsf{Cyl}(F):\mathrm{Core}(\Delta[n]\times\Delta[1])\to \fgt\hat{A}'_\bullet$ such that $\mathsf{Cyl}(F)([n], 0)=F$ and $\mathsf{Cyl}(F)([n],1)=R(F)\in \Phi(\fgt\hat{A}_n)$.
\end{lemma}
\begin{proof}
   Let $F\in \fgt{A}_\bullet'$. The natural transformation $\eta_F:F\to R(F)$ determines a simplicial map $\eta_F:\tilde{\Delta}^n\times \Delta[1]^{op}\to \scr{NC}_{A,w}^{op}$. Since $\tilde{\Delta}^1\cong \scr{N}(0\to 01\xleftarrow{}1)$, define
    \begin{align*}
   \tilde {\mathsf{Cyl}}(F):\tilde{\Delta}^n\times \tilde\Delta^1\cong \mathrm{sd}\Delta[n]\times\mathrm{sd}\Delta[1] \to \scr{NC}_{A,w}^{op}
    \end{align*}such that $\tilde {\mathsf{Cyl}}(F)|_{\tilde{\Delta}^n\times \Delta[1]^{op}}=\eta_F$ and $\tilde {\mathsf{Cyl}}(F)|_{\tilde{\Delta}^n\times \Delta[1]}=id_F$. 
    Note that there is a canonical simplicial map 
\[\mathrm{sd}(\Delta[n]\times\Delta[1])\to \mathrm{sd}\Delta[n]\times\mathrm{sd}\Delta[1]
\]Therefore, a map $\tilde{\Delta}^n\times \tilde{\Delta}^1\to \scr{NC}_{A,w}^{op}$ determines a map map $\mathrm{sd}(\Delta[n]\times\Delta[1])\to \scr{NC}_{A,w}^{op}$ from precomposition. By the first part of Lemma~\ref{coreandsd}, there is a corresponding semisimplicial map $\mathrm{Core}(\Delta[n]\times\Delta[1])\to \fgt\hat{A}'_\bullet$.
    
Therefore, the simplicial map $ \tilde {\mathsf{Cyl}}(F)$ determines a map $\mathsf{Cyl}(F):\mathrm{Core}(\Delta[n]\times\Delta[1])\to \fgt\hat{A}'_\bullet$, and the second part of Lemma~\ref{coreandsd} shows that $\mathsf{Cyl}(F)$ satisfies our requirement.
\end{proof}

Let\[d^i_{\mathsf{Cyl}}:=d^i\times id_{\Delta[1]}:\Delta[n-1]\times \Delta[1]\to \Delta[n]\times \Delta[1]\] be the simplicial map induced from the inclusion of the $i$th face, $d^i:\Delta[n-1]\to \Delta[n]$. Because $\eta_F|_{d^i(\tilde \Delta[n-1])}=\eta_{F\circ d^i}$,  we have for any $F\in \fgt{A}_\bullet'$, 
    \begin{equation}
    \label{CylrespFaces}
        \mathsf{Cyl}(F)\circ \mathrm{Core}(d^i_\mathsf{Cyl})=\mathsf{Cyl}(F\circ d^i)
    \end{equation}
\begin{lemma}
\label{surjective}For any $\gamma:S^n\to \|\fgt\hat{A}'_\bullet\|$, $\gamma$ is homotopic to a $\tilde\gamma:S^n\to \|\fgt\hat{A}'_\bullet\|$ that factors through $\|\Phi\|$.
\end{lemma}
\begin{proof}
    Any continuous map between CW complexes is homotopic to a CW map. Therefore, it suffices to consider only the CW cases. We have for any CW map $\gamma:S^n\to \|\fgt\hat{A}'_\bullet\|$, $S^n$ factors through the $n$-skeleton $\|\fgt\hat{A}'_\bullet\|_n$, which is the fat realization of the $k$-simplices of $\fgt\hat{A}'_\bullet$ for $k\leq n$.

    The idea of the proof is to use the result from Lemma~\ref{cylinderA}, which allows us to construct a homotopy that passes a map whose image lies in the bigger space $\|\fgt\hat{A}'_\bullet\|$ to the subspace $\|\fgt\hat{A}_\bullet\|$ (see Fig.~\ref{fig:ProofIdea}).

\begin{figure}
        \centering
        \includegraphics[width=0.5\linewidth]{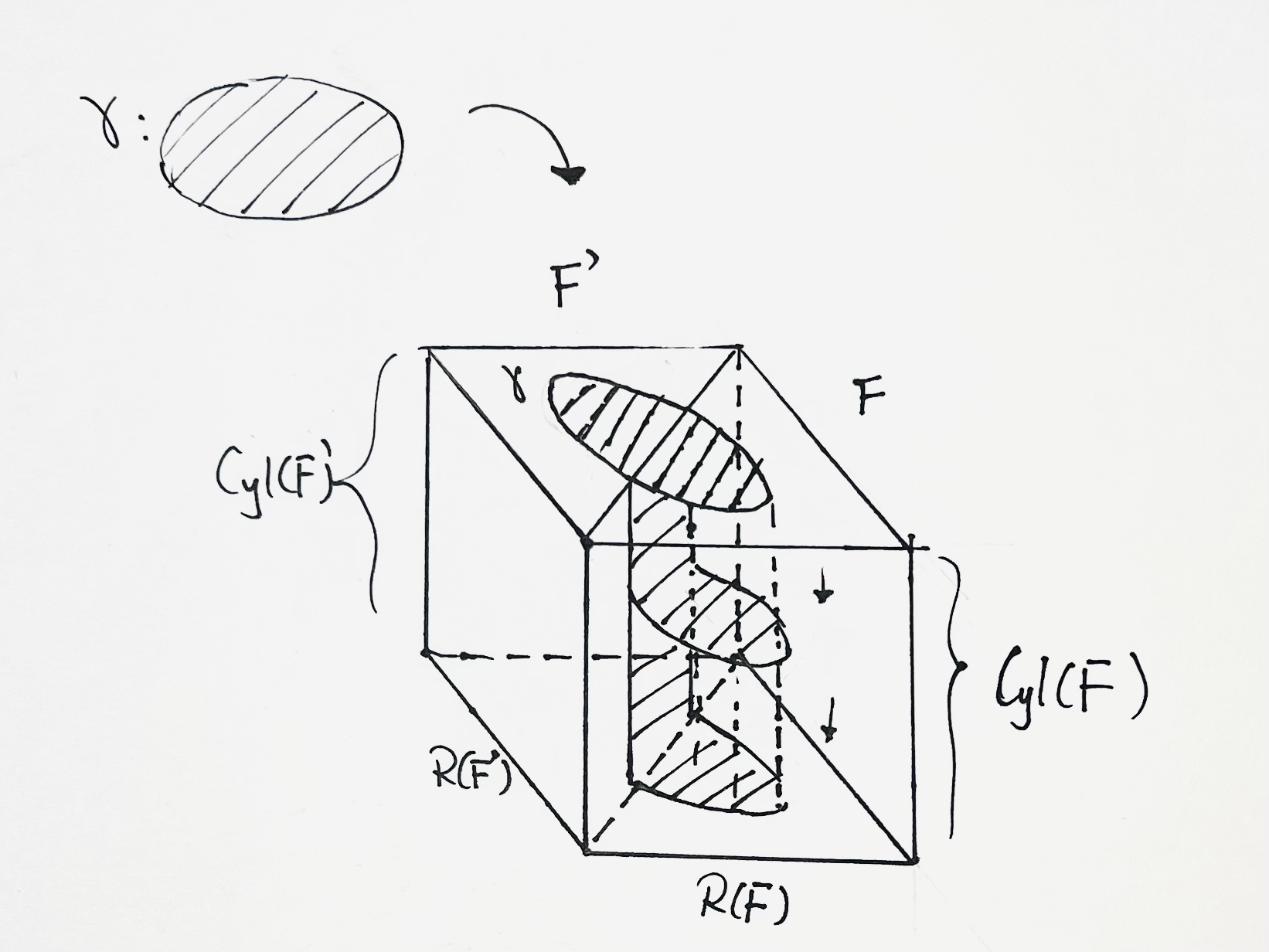}
        \caption{A Figure Illustrating the idea of the proof of Lemma~\ref{surjective}. Let $F$ and $F'\in \fgt\hat{A}'_2$ sharing a face, and let $\gamma$ be a map factoring through $F$ and $F'$. We see that the realization of $\mathsf{Cyl}(F)$ allows us push $\gamma$ onto $R(F)\in \fgt\hat{A}_2$}
        \label{fig:ProofIdea}
    \end{figure}
    We now formally construct the proof.
        
    For any closed set of $\|\fgt\hat{A}'_\bullet\|_n$ corresponding to a $n$-simplex $F$, we will abuse notation to denote it as $F\subseteq \|\fgt\hat{A}'_\bullet\|_n$ as well. For any such $F\subseteq \|\fgt{A}'_\bullet\|_n$, there is a cylinder $\mathsf{Cyl}(F)$ in $\|\fgt\hat{A}'_\bullet\|_{n+1}$. This determines a homotopy of $\gamma$ to a map $\gamma'$, where $\gamma'|_{S^n\ \backslash \ \gamma^{-1}(F)}=\gamma|_{S^n\ \backslash \ \gamma^{-1}(F)}$ and $\gamma'|_{\gamma^{-1}(F)}$ live in the realization of $\mathsf{Cyl}(F)\ \backslash \ \mathrm{int}\mathsf{Cyl}(F)\cup\mathrm{int}(F)$. Here $\mathrm{int} U$ denotes the interior of the set $U\subseteq X$ for some space $X$.
    
    For $F\neq F' \in \fgt\hat{A}_n'$, $\Big[\mathrm{int}\mathsf{Cyl}(F)\cup\mathrm{int}(F)\Big]\cap \Big[\mathrm{int}\mathsf{Cyl}(F')\cup\mathrm{int}(F')\Big]=\emptyset$. Therefore, we have shown that $\gamma$ is homotopic to a map $\gamma'_{n-1}$ such that  
    \[
 \gamma'_{n-1}:S^n\to   \bigcup_{F\in \fgt\hat{A}'_n}\{\mathsf{Cyl}(F)\ \backslash \ \mathrm{int}\mathsf{Cyl}(F)\cup\mathrm{int}(F)\}\cup \|\fgt\hat{A}_\bullet\|_{n}\hookrightarrow\|\fgt{A}'_\bullet\|_n\]
 
       By Equation~\ref{CylrespFaces}, the construction of $\mathsf{Cyl}(F)$ respects the semisimplicial structure of $\fgt{\hat A}'_\bullet$, we have 
       \[
       \mathsf{Cyl}(F)\ \backslash \ \mathrm{int}\mathsf{Cyl}(F)\cup\mathrm{int}(F)=\bigcup_{0\leq i\leq n}     \mathsf{Cyl}(F\circ d^i)
       \]and therefore,
       \begin{align*}
              \bigcup_{F\in \fgt\hat{A}'_n}\{\mathsf{Cyl}(F)\ \backslash \ \mathrm{int}\mathsf{Cyl}(F)\cup\mathrm{int}(F)\}\cup \|\fgt\hat{A}_\bullet\|_{n}=  \bigcup_{F\in \fgt\hat{A}'_n, 0\leq i\leq n}\{\mathsf{Cyl}(F\circ d^i )\}\cup \|\fgt\hat{A}_\bullet\|_{n}
              \\ \subseteq \bigcup_{G\in \fgt\hat{A}'_{n-1}}\{\mathsf{Cyl}(G )\}\cup \|\fgt\hat{A}_\bullet\|_{n}
       \end{align*}
    We may use the same argument as above to construct a homotopy between $\gamma_{n-1}'$ and a map
       \begin{align*}
        \gamma'_{n-2}:S^n\to    \bigcup_{G\in \fgt\hat{A}'_{n-2}, 0\leq i\leq n}\{\mathsf{Cyl}(G) )\}\cup \|\fgt\hat{A}_\bullet\|_{n}\subseteq \|\fgt\hat{A}'_\bullet\|_n
       \end{align*}Since $n$ is finite, this process would stop eventually until we reach $\gamma'_0$ that factors through \begin{align*}
           \bigcup_{G\in \fgt\hat{A}'_{0}}\{\mathsf{Cyl}(G )\}\cup \|\fgt\hat{A}_\bullet\|_{n}
       \end{align*} But $\fgt\hat{A}'_{0}= \fgt\hat{A}_{0}$, so the above space is precisely $\|\fgt{A}_\bullet\|$ and our proof is complete.
\end{proof}
       
We have shown that $\Phi_*:[S^n, \|\fgt\hat{A}_\bullet\|]\to[S^n, \|\fgt\hat{A}'_\bullet\|]$ is surjective. Recall that $R\circ \Phi=id_{\fgt \hat A}$, $\Phi_*$ is also injective. Thus, there is the following
\begin{proposition}
\label{fourth}
    The semisimplicial map $\Phi:\fgt\hat{A}_\bullet\to \fgt\hat{A}'_\bullet$ induces a weak homotopy equivalence between fat realizations.
\end{proposition}

We return to the diagram of the zigzags of maps in the beginning of this section, Diagram~\ref{zigzag}.
    \begin{center}
    \begin{tikzcd}
B\Aut(A) \arrow[rd,"\text{Prop. } \ref{first}"] &                                          & {{|\scr{N}^{\mathrm{}}\scr{C}_{A,w}|}\simeq {|\scr{N}^{\mathrm{}}\scr{C}^{op}_{A,w}|}} \arrow[ld,"\text{Prop. }\ref{second}"] \arrow[rd, "\text{Prop. } \ref{third}"] &                                              & {\|\fgt\hat{A}_\bullet\|\simeq \hat{A}} \\
                    & {{|\scr{N}^{\mathrm{hc}}\scr{C}_{A,w}|}} &                                                                                                              & {|\hat{A}'_\bullet|\simeq \|\fgt\hat{A}'_\bullet\|} \arrow[ru,"\text{Prop. } \ref{fourth}"] &                                
\end{tikzcd}
    \end{center} Each of the maps in this diagram are induced by a simplicial weak equivalence or a semisimplicial map that induces a weak equivalence.

    \subsection{Classification of Homogeneous Functors in Manifold Calculus}
    \label{section46}
    The goal of his section is to show Theorem~\ref{class} stated in the introduction. Let $\scr{M}$ be a simplicial model category.
\begin{definition}
    \label{metadef} Given a smooth manifold $M$, a functor $F:\scr{O}(M)^{op}\to \scr{M}$ is said to be good if the following two conditions are satisfied,
    \begin{itemize}
        \item If $U\subseteq V\in \scr{O}(M)$ is an isotopy equivalence, then $F(V)\to F(U)$ is a weak equivalence.
        \item Given $U_i\subseteq U_{i+1}\in\scr{O}(M)$, the canonical map $F(\cup_{i}U_i)\to \mathrm{holim}_iF(U_i)$ is a weak equivalence.
    \end{itemize}

    A good functor $F:\scr{O}(M)^{op}\to \scr{M}$ is polynomial of degree $\leq k$  for some $k\in \mathbb{N}$, if given any open set $U\in \scr{O}(M)$, for any collection of pairwise disjoint nonempty closed sets $A_0,...,A_k\subset U$, the canonical map
    \begin{align*}
        F(U)\to \mathrm{holim}_{\emptyset\neq S\subseteq \{0,...,k\}}F(U-\cup_{i\in S}A_i)
    \end{align*} is a weak equivalence.

    Let $B_k(M)$ be the full subcategory $\scr{O}(M)$ whose objects are the open sets that are diffeomorphic to at most $k$ disjoint open balls. 

    For any good functor $F:\scr{O}(M)^{op}\to \scr{M}$, there is a polynomial functor of degree $\leq k$, $\scr{T}_kF:\scr{O}(M)^{op}\to \scr{M}$ defined objectwise for any $U\in \scr{O}(M)$ as
    \begin{align*}
        \scr{T}_kF(U):=\underset{V\in B_k(U)}{\mathrm{holim}}F(V)
    \end{align*}
\end{definition}

We now define homogeneous functors of degree $k$.
\begin{definition}
    \label{homog}A functor $F:\scr{O}(M)^{op}\to \scr{M}$ is homogeneous of degree $k$, if 
    \begin{itemize}
        \item $F$ is a polynomial of degree $\leq k$, and 
        \item $\scr{T}_{k-1}F\to *$ is a weak equivalence. Here, $*\in \scr{M}$ is the terminal object of $\scr{M}$.
    \end{itemize}
\end{definition}

    Let $\scr{F}_{kA}(\scr{O}(M)^{op}, \scr{M})$ denote the full subcategory of $\Fun(\scr{O}(M)^{op},\scr{M})$ whose objects are the homogeneous functors of degree $k$ such that $F(U)\simeq A$ for any $U\in \scr{O}(M)$ that is diffeomorphic to the disjoint union of $k$ disjoint open balls. Recall the relation $\mathfrak{w}$ in Section~\ref{section42} on $\scr{F}_{kA}(\scr{O}(M)^{op},  \scr{M})$, where $F\mathfrak{w}F'$ if there is a zig zag of objectwise weak equivalences between $F$ and $F'$.

\begin{theorem}
   \label{tsopmene2024} \cite[Theorem 1.1]{tsopmene2024classificationhomogeneousfunctorsmanifold}
Given any fibrant-cofibrant object $A\in \scr{M}$, the set 
$\scr{F}_{kA}(\scr{O}(M)^{op}, \scr{M})/\frak{w}$ can be described as follows:
    \begin{itemize}
        \item If $k=1$, then $\scr{F}_{1A}(\scr{O}(M)^{op}, \scr{M})/\mathfrak{w}\cong [M, \hat{A}]$
        \item If $k>1$ and $\scr{M}$ is pointed, then $\scr{F}_{kA}(\scr{O}(M)^{op}, \scr{M})/\mathfrak{w}\cong [F_k(M), \hat{A}]$
    \end{itemize}
\end{theorem}

\begin{theorem}[Theorem~\ref{class}]
Given any fibrant-cofibrant object $A\in \scr{M}$, 
    \begin{itemize}
        \item If $k=1$, then $\scr{F}_{1A}(\scr{O}(M)^{op}, \scr{M})/\mathfrak{w}\cong [M, B\Aut(A)]$
        \item If $k>1$ and $\scr{M}$ is pointed, then $\scr{F}_{kA}(\scr{O}(M)^{op}, \scr{M})/\mathfrak{w}\cong [F_k(M), B\Aut(A)]$
    \end{itemize}
\end{theorem}

\begin{proof}
    From Theorem~\ref{main} we have that $\hat{A}\simeq B\Aut(A)$, then the statement follows from Theorem~\ref{tsopmene2024}.
\end{proof}
    
\section{When \protect\boldmath$\scr{M}=\Top$}\label{section5}

In this section, we will show that are classification result (Theorem~\ref{class}) is equivalent to that of Weiss in \cite{Weiss_1999}, for $\scr{M}=\Top$, the category of topological spaces. 

\begin{definition}
    \cite[Example 7.1]{Weiss_1999} Given a fibration $p:Z\to F_k(M)$, define $\scr{F}_p:\scr{O}(M)^{op}\to \Top$ as follows: let $\Gamma(p,V)$ denote the space of partial sections of $V\subseteq F_k(M)$, for any $U\in \scr{O}(M)$, 
    \begin{align*}
        \scr{F}_p(U):=\Gamma(p; F_k(U))
    \end{align*} 
\end{definition}
This functor is a polynomial functor of degree $\leq k$. Intuitively, this is because the space of sections of a fibration over an open set $U$ is obtained by gluing sections over smaller open sets that covers $U$ along the intersections. For the case of $F_k(M)$, if we wish to find a covering of $V_i\hookrightarrow F_k(U)$, such that $V_i=F_k(U_i)$ for some $U_i\subseteq U$, then we will need at least $k$ such open subsets.

Denote, for any $U\in \scr{O}(M)$, $\triangle_k(U):=(U^k/\Sigma_k)\backslash F_k(U)$. Let $Q(U)\subseteq \scr{O}(U^k/\Sigma_k)$ such that for any $V\in Q(U)$, $V\cap \triangle_k(U)\neq \emptyset$. 
\begin{definition}
     \cite[Lemma 7.2]{Weiss_1999}Given a fibration $p:Z\to F_k(M)$, define $\scr{G}_p:\scr{O}(M)^{op}\to \Top$ as follows: 
     \begin{align*}
         \scr{G}_p(U):=\mathrm{hocolim}_{V\in Q(U)}\Gamma(p, V\cap F_k(U))
     \end{align*}
     This is a good functor. Moreover, this is a polynomial functor of degree $\leq k-1$\cite[Proposition 7.5]{Weiss_1999}.
\end{definition}
The functor $\scr{G}_p$ is polynomial of degree $\leq k-1$ because when our open set $F_k(U)$ of $F_k(M)$ are very close to $\triangle_k(M)$, we only need at most $k-1$ open subsets of $U$ to cover $F_k(U)$, because they sufficiently close to each other. 

Then, there is an inclusion of functors  $\eta_p:\scr{F}_p\to \scr{G}_p$. In fact, this $\eta_p$ is exhibits $\scr{G}_p$ as a polynomial functor of degree $\leq k-1$ approximation of $\scr{F}_p$ \cite[Proposition 7.6]{Weiss_1999}. 

Suppose that we fix a point $s\in \scr{G}_p(M)$, then we can construct a functor $\scr{E}_{p,s}$ as follows: for any $U\in \scr{O}(M)$
\begin{align*}
   \scr{E}_{p,s}(U):=\mathrm{hofib}(\scr{F}_p(U)\to \scr{G}_p(U))
\end{align*}
This is then a homogeneous functor of degree $ k$.

Let $S\in F_k(M)$ and let $U_S\in \scr{O}(M)$ be a choice of tubular neighbourhood of the points in $M$ corresponding to $S$. Then,  there is a weak homotopy equivalence $\scr{E}_{p,s}(U_S)\simeq p^{-1}(S)$. In fact, there is the following result.
\begin{theorem}
    \cite[Theorem 8.5]{Weiss_1999} Given a space $A\in \Top$, for any $E\in \scr{F}_{kA}(\scr{O}(M)^{op},\Top)$, there is a fibration $p:Z\to F_k(M)$ such that for any $S\in F_k(M)$, $p^{-1}(S)\simeq A$, such that $E\simeq \scr{E}_{p,s}$.
\end{theorem}

In other words, $\scr{F}_{kA}(\scr{O}(M)^{op},\Top)/\frak{w}$ is in bijection with the equivalence classes of fibration $p:Z\to F_k(M)$ with each fiber homotopic to $A$.

We now recall some results on the classification of fibration over a space $X\in \Top$.

\begin{theorem}
    \cite[Corollary 9.5]{MayFib} Given a cofibrant topological space $A\in \Top$ (a CW complex), the equivalence classes of fibration over $X$ such that every fiber has the homotopy type of $A$ is then in bijection with $[X, B\Aut(A)]$
\end{theorem}

Therefore, if our homogeneous functor has corresponding fibration $p:Z\to F_k(M)$ such that its fiber is homotopic to $A$, then it corresponds to a homotopy class of continuous map $F_k(M)\to B\Aut(A)$, and vice versa. This shows that our result Theorem~\ref{class} is a generalization of Weiss's result for any simplicial model category $\scr{M}$. 
\section{The ∞-Category Theoretic Analogue}\label{section6}

It is natural to describe the Manifold Calculus of Functors using the language of ∞-categories, which is developed by Arakawa in \cite{arakawa2026contextmanifoldcalculus}. Let $\scr{M}$ be an ∞-category that admits finite and tower limits. We have the following definition which is completely analogous with Definition~\ref{metadef}.
\begin{definition}
    Given a smooth manifold $M$, a functor $F:\scr{O}(M)^{op}\to \scr{M}$ is said to be good if the following two conditions are satisfied:
    \begin{itemize}
        \item If $U\subseteq V\in \scr{O}(M)$ is an isotopy equivalence, then $F(V)\to F(U)$ is an isomorphism.
        \item Given $U_i\subseteq U_{i+1}\in \scr{O}(M)$, the canonical map $F(\cup_{i}U_i)\to \mathrm{lim}_iF(U_i)$ is an isomorphism.
    \end{itemize}

   Given a good functor $F:\scr{O}(M)^{op}\to \scr{M}$, $F$ is a polynomial functor of degree $\leq k$ for some $k\in \mathbb{N}$,, if given any open set $U\in \scr{O}(M)$, for any collection of pairwise disjoint nonempty closed sets $A_0,...,A_k\subset U$, the canonical map
    \begin{align*}
        F(U)\to \mathrm{lim}_{\emptyset\neq S\subseteq \{0,...,k\}}F(U-\cup_{i\in S}A_i)
    \end{align*} is an isomorphism.
\end{definition} 
Let $\Fun_{\scr{I}\mathrm{so}}(\scr{O}(M)^{op}, \scr{M})$ be the full subcategory of $\Fun_{}(\scr{O}(M)^{op}, \scr{M})$ whose objects are the good functors, and let $\scr{P}\mathrm{oly}_k(\scr{O}(M)^{op},\scr{M})$ be the full subcategory of $\Fun_{\scr{I}\mathrm{so}}(\scr{O}(M)^{op}, \scr{M})$ whose objects are the polynomial functors of degree $\leq k$. Then, the following was shown in \cite{arakawa2026contextmanifoldcalculus}
\begin{theorem}
    \cite[Theorem 1.3 (op)]{arakawa2026contextmanifoldcalculus} The right Kan extension along the restriction $B_k(M)\hookrightarrow\scr{O}(M)$ determines a categorical equivalence of ∞-categories:
    \begin{align*}
        \Fun_{\scr{I}\mathrm{so}}(B_k(M)^{op}, \scr{M})\xrightarrow{\simeq}\scr{P}\mathrm{oly}_k(\scr{O}(M)^{op},\scr{M})
    \end{align*}
    In particular, this induces a right adjoint\begin{align*}
        \scr{T}_k:\Fun_{\scr{I}\mathrm{so}}(\scr{O}(M)^{op}, \scr{M})\to \scr{P}\mathrm{oly}_k(\scr{O}(M)^{op},\scr{M})
    \end{align*}
\end{theorem}

The above results of $\scr{T}_k$ agree with our definition of the polynomial approximation functor $\scr{T}_k$ in Definition~\ref{metadef}

\begin{definition}
    Given a functor $F:\scr{O}(M)^{op}\to \scr{M}$, $F$ is homogeneous of degree $k$, if 
    \begin{itemize}
        \item $F$ is a polynomial functor of degree $\leq k$
        \item $\scr{T}_{k-1}F\to *$ is an isomorphism. Here, $*\in \scr{M}$ is the terminal object of $\scr{M}$.
    \end{itemize}
\end{definition}
Let $\scr{H}\mathrm{omog} _k(\scr{O}(M)^{op},\scr{M})$ denote the full subcategory of $\scr{P}\mathrm{oly}_k(\scr{O}(M)^{op},\scr{M})$ whose objects are the homogeneous functors of degree $k$. We have the following theorem:
\begin{theorem}
    For $k=1$ or $k>1$ and $\scr{M}$ is pointed, the restriction functor induces a categorical equivalence
    \begin{align*}
        \scr{H}\mathrm{omog} _k(\scr{O}(M)^{op},\scr{M})\to \Fun_{\scr{I}\mathrm{so}}(\scr{I}^{(k), op},\scr{M})
    \end{align*} Here $\scr{I}^{(k)}\subset B_k(M)$ is a subcategory of open sets with exactly $k$ connected components and only isotopy equivalence.
\end{theorem}
\begin{proof}
    The case for $k>1$ and $\scr{M}$ is pointed is proved in \cite[Proposition 3.2]{arakawa2026contextmanifoldcalculus}. We  then provide a simple proof for the case of $k=1$ here.
    Notice that a functor $F:\scr{O}(M)^{op}\to \scr{M}$ is homogeneous of degree $1$ if and only if it is polynomial of degree $\leq 1$ and $F(\emptyset)\simeq *$ (in other words, reduced), so there is
    \begin{align*}
            \scr{H}\mathrm{omog} _1(\scr{O}(M)^{op},\scr{M})\simeq \scr{P}\mathrm{oly}_1^\mathrm{red}(\scr{O}(M)^{op},\scr{M})
    \end{align*}Here the superscript $\mathrm{red}$ refers to the reduced functors.

    Since $\scr{P}\mathrm{oly}_1^\mathrm{red}(\scr{O}(M)^{op},\scr{M})\simeq\Fun^\mathrm{red}_{\scr{I}\mathrm{so}}(B_1(M)^{op},\scr{M}) $, and any inclusion of $U\subset V$ in $B_1(M)^{op}$ is an isotopy equivalence if $U\neq \emptyset$, and there is
    \begin{align*}
       \scr{P}\mathrm{oly}_1^\mathrm{red}(\scr{O}(M)^{op},\scr{M})\simeq \Fun_{\scr{I}\mathrm{so}}(\scr{I}^{(1), op},\scr{M})
    \end{align*}
\end{proof}

Here appears the benefit of working in ∞-categories, by definition that given any $F\in \Fun_{\scr{I}\mathrm{so}}(\scr{I}^{(k), op},\scr{M})$, for any morphism $\iota\in \scr{I}^{(k)}$, $F(\iota)$ is an isomorphism. Therefore,
\begin{align*}
    \Fun_{\scr{I}\mathrm{so}}(\scr{I}^{(k), op},\scr{M})^\simeq =\Fun_{}(\scr{I}^{(k), op},\scr{M}^\simeq )
\end{align*} Here, $\scr{M}^\simeq$ denotes the maximal Kan complex contained in $\scr{M}$.
\begin{proposition}
        For $k=1$ or $k>1$ and $\scr{M}$ is pointed, there is homotopy equivalence
          \begin{align*}
          \scr{H}\mathrm{omog} _k(\scr{O}(M)^{op},\scr{M})^\simeq \to \Fun_{\scr{I}\mathrm{so}}(\scr{I}^{(k), op},\scr{M}^\simeq )
    \end{align*}
\end{proposition}

It is shown in \cite[Lemma 3.5]{Weiss_1999} that $|\scr{I}^{(k)}|\simeq F_k(M)$. Therefore, we have the following.
\begin{theorem}
    $ \scr{H}\mathrm{omog} _k(\scr{O}(M)^{op},\scr{M})^\simeq\simeq \Fun(\mathrm{sing}F_k(M),\scr{M}^\simeq) $
\end{theorem}

\begin{corollary}
    Given $A\in \scr{M}$, let $\scr{F}_{kA}(\scr{O}(M)^{op},\scr{M})$ denote the ∞-category of homogeneous functors of degree $k$, such that for object $U\in \scr{I}^{(k)}$, $F(U)\simeq A$, then for $k=1$ or $k>1$ and $\scr{M}$ is pointed
    \begin{align*}
        \scr{F}_{kA}(\scr{O}(M)^{op},\scr{M})^\simeq \simeq \Fun(\mathrm{sing}F_k(M),B\Aut(A))
    \end{align*} Here $\Aut(A)$ is the simplicial monoid of isomorphisms of $A$ in $\scr{M}$, obtained from $\mathfrak{C}\scr{M}^\simeq$ and is weakly equivalent to $\Map_{\scr{M^\simeq}}(A,A)$.
\end{corollary}
\begin{proof}
    The Kan complex $  \scr{F}_{kA}(\scr{O}(M)^{op},\scr{M})^\simeq$ is equivalent to the subsimplicial set of $ \Fun(\mathrm{sing}F_k(M),\\\scr{M}^\simeq)$ that factors through $\scr{M}_A^\simeq \to \scr{M}^\simeq$. Here, $\scr{M}_A^\simeq$ denotes the connected component in $\scr{M}^\simeq$ of $A$.
    
     There are categorical equivalences of ∞-categories, $\scr{M}^\simeq \to \scr{N}^\mathrm{hc}\mathfrak{C}\scr{M}^{\simeq}$ and $\scr{M}^\simeq_A \to \scr{N}^\mathrm{hc}\mathfrak{C}\scr{M}^{\simeq}_A$. It is easy to see that $\scr{C}_{\Map_{\mathfrak{C}\scr{M}_A^\simeq }(A,A)}\to \mathfrak{C}\scr{M}_A^\simeq $ is a categorical equivalence of locally Kan categories. Therefore, we have $\scr{M}_A^\simeq \simeq B\Map_{\mathfrak{C}\scr{M}^\simeq }(A,A)$. Lastly, we have $\Map_{\mathfrak{C}\scr{M}^\simeq }(A,A)\simeq \Map_{\scr{M^\simeq}}(A,A)$.
\end{proof}
\subsection{Connection with Our Results in Simplicial Model Categories}

Recall that there is a natural way to interpret a simplicial model category $\scr{M}$ as an ∞-category, which takes the homotopy coherent nerve of its full subcategory whose objects are the fibrant-cofibrant object, $\scr{N}^{\mathrm{hc}}\scr{M}^{cf}$. The following is a trivial result.
\begin{lemma}
    For any category $K$, the function
    \begin{align*}
        \Hom(K,\scr{M})/\mathfrak{w}\xrightarrow{QR}\Hom(K, \scr{M}^{cf})/\frak{w}
    \end{align*} is bijective. Here $F\mathfrak{w}F'$ if they are connected by a zigzag of objectwise weak equivalence, and $QR$ denote the objectwise fibrant-cofibrant replacement of functors.
\end{lemma}
\begin{lemma}
    Given $F\in \Fun(\scr{O}(M)^{op}, \scr{M})$, $F$ is a polynomial functor of degree $\leq k$ if and only if $\scr{N}^{\mathrm{hc}}QRF$ is a polynomial functor of degree $\leq k$ in $\Fun(\scr{O}(M)^{op}, \scr{N}^{\mathrm{hc}}\scr{M}^{cf})$.
\end{lemma}
\begin{proof}
    From \cite[Theorem 4.2.4.1]{lurieHTT}, if $G:K^\triangleleft\to \scr{M}^{cf}$ determines a homotopy limit diagram if and only if $\scr{N}^\mathrm{hc}G:\scr{N}K^{\triangleleft}\to \scr{N^\mathrm{hc}M}^{cf}$ determines a limit diagram. It then follows from the characterization of polynomial functors in both the simplicial model category and the ∞-category that the two definitions coincide.
\end{proof}
\begin{lemma}
    Given a locally Kan simplicial category $\scr{M}$, for any ordinary category $K$, there is bijection:
    \begin{align*}
       \scr{N^{\mathrm{hc,*}}}: \Fun(K, \scr{M})/\mathfrak{w}\to \pi_0\Fun(\scr{N}K, \scr{N^{\mathrm{hc}}}\scr{M} )^\simeq 
    \end{align*}Here $F\mathfrak{w}F'$ if there is a zigzag of objectwise homotopy equivalence.
\end{lemma}
\begin{proof}
    We first show that the above determines a function. 
    
    A natural transformation $\eta:F\to F'$ is a objectwise homotopy equivalence when the vertical morphisms of the corresponding functor $\eta':K\times I\to \scr{M}$ factors through the connected components of the mapping space of $\scr{M}$ of homotopy equivalence. This induces $\scr{N}^\mathrm{hc}\eta':\scr{N}K\times \scr{N}I\cong\scr{N}K\times \Delta[1]\to \scr{N^{\mathrm{hc}}}\scr{M} $. Then by \cite[Theorem 4.4.4.4]{kerodon}, $\scr{N}^\mathrm{hc}\eta'$ is an isomorphism of functors if and only if for any $k\in \scr{N}K$, $\scr{N}^\mathrm{hc}\eta'(k)$ is an isomorphism. Then by Lemma~\ref{lemma1}, this is true for $\scr{N}^\mathrm{hc}\eta'$, so $\scr{N}^\mathrm{hc}F$ and $\scr{N}^\mathrm{hc}F'$ lives in the same connected component of $\Fun(\scr{N}K, \scr{N^{\mathrm{hc}}}\scr{M} )^\simeq $.

    Since a morphism in $\scr{M}$ is a homotopy equivalence if and only if it is an isomorphism in $\scr{N}^{\mathrm{hc}}\scr{M}$, $\scr{N^{\mathrm{hc,*}}}$ is injective. 
    
    There is $\Hom(\scr{N}K, \scr{N^{\mathrm{hc}}}\scr{M})\cong \Hom(\mathfrak{C}\scr{N}K,\scr{M})$. To show surjectivity, it suffices to show that for any $F:\mathfrak{C}\scr{N}K\to \scr{M}$, there is a objectwise weak equivalence $\eta:F\to F'$ such that $F':\mathfrak{C}\scr{N}K\to \scr{M}$ factors through $\epsilon:\mathfrak{C}\scr{N}K\to K$. It was shown in \cite[Page 66]{Cordier_Porter_1986}\begin{align*}
\epsilon^*: \mathrm{Ho}\Fun(K, \scr{M})\to      \mathrm{Ho}\Fun(\mathfrak{C}\scr{N}K , \scr{M}) 
    \end{align*}is a categorical equivalence for any locally Kan category $\scr{M}$, which implies what we want to prove.
\end{proof}
Therefore, we can safely translate the work of Manifold Calculus of Functors in ∞-category to Manifold Calculus in simplicial model category. Especially, there is the following.

\begin{proposition}(Theorem~\ref{class})
    For $k=1$ or $k>1$ and $\scr{M}$ is pointed
    \begin{align*}
        \scr{F}_{kA}(\scr{O}(M)^{op},\scr{M})/\mathfrak{w}\cong[F_k(M),B\Aut(A)]
    \end{align*}
\end{proposition}

\bibliographystyle{plainnat} 
\bibliography{bibliography} 

@article{Rourkesemisimplicialset,
	author = {C. Rourke and B. Sanderson},
	journal = {The Quarterly Journal of Mathematics},
	title = {{$\Delta$}-{Set I}: Homotopy Theory},
	year = {1971}}

@book{Hatcher:478079,
	address = {Cambridge},
	author = {Hatcher, A.},
	publisher = {Cambridge Univ. Press},
	title = {{Algebraic Topology}},
	year = {2000}}

@article{WhiteheadI,
	author = {Whitehead, J. H. C.},
	journal = {Bull. Am. Math. Soc.},
	title = {Combinatorial Homotopy, I},
	year = {1949}}

@book{goerss-jardine,
	author = {Goerss, P. G. and Jardine, J. F.},
	publisher = {Birkh{\"a}user Basel},
	title = {Simplicial homotopy theory},
	year = 2009}

@article{tsopmene2024classificationhomogeneousfunctorsmanifold,
	author = {Songhafouo Tsopm{\'e}n{\'e}, P. A. and Stanley, D.},
	journal = {Journal of Homotopy and Related Structures},
	number = {1},
	pages = {63--103},
	title = {Classification of homogeneous functors in manifold calculus},
	volume = {20},
	year = {2025}}

@article{JOYAL2002207,
	author = {Joyal, A.},
	doi = {https://doi.org/10.1016/S0022-4049(02)00135-4},
	journal = {Journal of Pure and Applied Algebra},
	number = {1},
	pages = {207-222},
	title = {Quasi-categories and Kan complexes},
	volume = {175},
	year = {2002}}

@inbook{QuillenGroupComp,
	author = {Quillen, D.},
	chapter = {Appendix Q},
	publisher = {Memoir of the A.M.S.},
	title = {On the group completion of a simplicial monoid},
	year = {1994}}

@book{lurieHTT,
	address = {Princeton, NJ},
	author = {Lurie, J.},
	isbn = {978-0-691-14048-3},
	publisher = {Princeton University Press},
	series = {Annals of Mathematics Studies},
	title = {Higher Topos Theory},
	volume = {170},
	year = {2009}}

@misc{arakawa2025classificationdiagramssimplicialcategories,
	author = {Arakawa, K.},
	howpublished = {arXiv:2401.16855},
	title = {Classification diagrams of simplicial categories},
	year = {2025}}

@article{Cordier1982DiagrammeHomotopique,
	author = {Cordier, J.},
	journal = {Cahiers de Topologie et G\'eom\'etrie Diff\'erentielle Cat\'egoriques},
	language = {French},
	number = {1},
	pages = {93--112},
	title = {Sur la notion de diagramme homotopiquement coh\'erent},
	volume = {23},
	year = {1982}}

@misc{kerodon,
	author = {Lurie, J.},
	howpublished = {\url{https://kerodon.net}},
	title = {Kerodon},
	year = {2026}}

@unpublished{lurieHA,
	author = {Lurie, J.},
	howpublished = {Preprint},
	title = {Higher Algebra},
	year = {2017}}

@mastersthesis{KBergEsd,
	author = {Berg, K.},
	school = {University of Oslo},
	title = {Edgewise Subdivision and Simple Maps},
	year = {2009}}

@article{Weiss_1999,
	author = {Weiss, M.},
	journal = {Geometry \& Topology},
	month = may,
	number = {1},
	pages = {67--101},
	publisher = {Mathematical Sciences Publishers},
	title = {Embeddings from the point of view of immersion theory: Part I},
	volume = {3},
	year = {1999}}

@book{MayFib,
	author = {May, J. P.},
	fseries = {Memoirs of the American Mathematical Society},
	publisher = {Providence, RI: American Mathematical Society (AMS)},
	series = {Mem. Am. Math. Soc.},
	title = {Classifying spaces and fibrations},
	volume = {155},
	year = {1975}}

@article{arakawa2026contextmanifoldcalculus,
	author = {Arakawa, K.},
	journal = {Journal of Homotopy and Related Structures},
	month = Mar,
	publisher = {Springer Science and Business Media LLC},
	title = {A context for manifold calculus},
	year = {2026}}

@article{Cordier_Porter_1986,
	author = {Cordier, J. and Porter, T.},
	journal = {Mathematical Proceedings of the Cambridge Philosophical Society},
	number = {1},
	pages = {65--90},
	title = {Vogt's theorem on categories of homotopy coherent diagrams},
	volume = {100},
	year = {1986}}

\end{document}